\documentclass[a4paper,10pt]{article}
\usepackage[utf8]{inputenc}

 \usepackage{tikz-cd} 
\usepackage{amsmath,amssymb,amsthm}
\usepackage[all]{xy}
\usepackage{graphicx}
\usepackage{color}
\usepackage[margin=1in]{geometry}
\usepackage[english]{babel} 
\usepackage{authblk}
\usepackage{amsmath,amssymb}  
\usepackage{amsthm}           

\theoremstyle{plain}                          
 
\theoremstyle{definition}                     

\newtheorem{lemma}{Lemma}[section]
\newtheorem{remark}[lemma]{Remark}
\newtheorem{example}[lemma]{Example}
\newtheorem{theorem}[lemma]{Theorem}
\newtheorem{corollary}[lemma]{Corollary}
\newtheorem{definition}[lemma]{Definition}
\newtheorem{proposition}[lemma]{Proposition}

\theoremstyle{remark}                         

\newcommand{\field}[1]{\mathbb{#1}} 
\newcommand{\C}{\mathcal{C}} 
\newcommand{\K}{\field{K}}
\renewcommand{\H}{\tilde{H}}

\renewcommand{\S}{\field{S}}
\renewcommand{\P}{\mathcal{P}}
\newcommand{\Sus}{\mathcal{S}}
\newcommand{\D}{\mathcal{D}}
\newcommand{\X}{\mathcal{X}}
\newcommand{\E}{\mathcal{E}}

\newcommand{\Q}{\field{Q}}
\newcommand{\NN}{\tilde{N}_*}
\newcommand{\N}{\tilde{N}^*}
\newcommand{\F}{\field{F}} 

\title{$\E_n$-Hopf invariants}
\author{Felix Wierstra}
\date{}
\affil{Faculty of Mathematics and Physics, Charles University, Sokolovsk\'a 49/83, 186 75 Praha 8, Czech Republic, felix.wierstra@gmail.com}

\begin{document}

\maketitle

\abstract{The classical Hopf invariant is an invariant of homotopy classes of maps from $S^{4n-1} $ to $S^{2n}$, and is an important invariant in homotopy theory.  The goal of this paper is to use the Koszul duality theory for $\E_n$-operads to define a generalization of the classical Hopf invariant. One way of defining the classical Hopf invariant is by defining a pairing between the cohomology of the associative bar construction on the cochains of a space $X$ and the homotopy groups of $X$. In this paper we will give a generalization of the classical Hopf invariant by defining a pairing between the cohomology of the $\E_n$-bar construction on the cochains of $X$ and the homotopy groups of $X$. This pairing gives us a set of invariants of homotopy classes of maps from $S^m$ to a simplicial set $X$, this pairing can detect more homotopy classes of maps than the classical Hopf invariant. 

The second part of the paper is devoted to combining the $\E_n$-Hopf invariants with the  Koszul duality theory for $\E_n$-operads to get  a relation between the $\E_n$-Hopf invariants of a space $X$ and the $\E_{n+1}$-Hopf invariants of the suspension of $X$. This is done by studying the suspension morphism for the $\E_\infty$-operad, which is a morphism from the $\E_{\infty}$-operad to the desuspension of $\E_\infty$-operad. We show that it induces a functor from $\E_\infty$-algebras to $\E_\infty$-algebras, which has the property that it sends an $\E_\infty$-model for a simplicial set $X$ to an $\E_\infty$-model for the suspension of $X$.

 We use this result to give a relation between the $\E_n$-Hopf invariants of maps from $S^m$ into $X$ and the $\E_{n+1}$-Hopf invariants of maps from $S^{m+1}$ into the suspension of $X$. One of the main results we show here, is  that this relation can be used to define invariants of stable homotopy classes of maps. 

\tableofcontents

\section{Introduction}




The classical Hopf invariant is defined as an invariant of maps $f:S^{4n-1} \rightarrow S^{2n}$ and was initially used to show that the Hopf fibration is not homotopic to the constant map, but found later many more applications. One way of defining the classical Hopf invariant is by using the associative bar construction on the cochains of a space $X$. More precisely this is done by defining a pairing
\[
\left<,\right>:H^*(B_{Ass} C^*(X;\K)) \otimes \pi_*(X) \rightarrow \K,
\]  
where $X$ is a simply-connected space and $\K$ is a commutative ring. The goal of this paper is to generalize this pairing by using the fact that the singular cochains on a space are not just an associative algebra, but are an $\E_{\infty}$-algebra. By using more of this $\E_{\infty}$-structure we will get more refined invariants which have the ability to detect more homotopy classes of maps. The idea is that we replace the associative bar construction by the  $\E_n$-bar construction, this bar construction has in general more homology and will therefore be able to detect more maps. The first main result we obtain in this paper is the following theorem which can be found as Theorem \ref{thrmEnhopfinv} below.

\begin{theorem}\label{theorem1}
Let $X$ be a simply-connected simplicial set, then for $k>n$ there exists a pairing 
\[
\left<,\right>_{\E_n}:H^k(B_{\E_n}\N(X))\otimes \pi_k(X)\rightarrow \K,
\]
where $B_{\E_n}\N(X)$ is the $\E_n$-bar construction on the reduced normalized chains of $X$ and $\K$ is a field.
\end{theorem}

The reason why we need the condition that $k>n$ is because the $n$-sphere is only formal as an $\E_n$-algebra for $n<k$, more details about this problem will be given in Section \ref{secEnhopfinv}. This is also the reason why we cannot use the full $\E_{\infty}$-structure to define invariants and need to restrict ourselves to the $\E_n$-structure.

The second part of the paper is spend to study the further properties of the $\E_n$-Hopf invariants. In particular we will study the the behavior of the $\E_n$-Hopf invariants with respect to suspensions. We do this by combining the $\E_n$-Hopf invariants with Fresse's theory of Koszul duality for $\E_n$-operads (see \cite{Fres3}) and the work by Berger and Fresse on operad actions on the normalized cochains of a simplicial set (see \cite{BF1}).

In \cite{BF1}, Berger and Fresse define a morphism $\sigma:\E_{\infty} \rightarrow \Sus^{-1} \E_{\infty}$, from the $\E_{\infty}$-operad to $\Sus^{-1} \E_{\infty}$, the desuspended $\E_{\infty}$-operad, called the suspension morphism. In Section \ref{secsuspensions}, we show how this suspension morphism relates the normalized chains on a space to the normalized chains on the suspension, this is summarized in the following theorem, which can be found as Theorem \ref{thrmsuspensionfunctor} below.

\begin{theorem}\label{theorem2}
Let $X$ be a simplicial set, then there is an isomorphism of $\E_{\infty}$-algebras between $\N(\Sigma X)$ and $\sigma^{-1} \N(X)$. Here $\N(\Sigma X)$ denotes the reduced normalized cochains of the (simplicial) reduced suspension of $X$ and $\sigma^{-1} \N(X)$ denotes the suspension of  the underlying cochain complex of $\N(X)$ equipped with the $\E_{\infty}$-algebra structure coming from the suspension morphism.
\end{theorem}

 The suspension morphism has the additional property that it restricts to a morphism from the $\E_{n+1}$-operad to $\Sus^{-1} \E_n$ and can be used to relate the $\E_n$-Hopf invariants to the $\E_{n+1}$-Hopf invariants. We can use this to define a map $\Upsilon:B_{\E_{n+1}}(\N(\Sigma X)) \rightarrow B_{\E_n}(\N(X))$. This map has the following property.
 
\begin{theorem}\label{theorem3}
Let $X$  be a simplicial set and let $f:S^k \rightarrow X$ be a map, where we assume that $k>n$. Let $\omega \in B_{\E_{n+1}}\N(\Sigma X)$, then we have the following relation between the Hopf pairings
\[
\left<\omega , \Sigma f \right>_{\E_{n+1}}=\left< \Upsilon \omega , f \right>_{\E_n},
\]
where $\Sigma f:S^ {k+1} \rightarrow \Sigma X$ is the suspension of $f$, and $\Upsilon $ is the map that is defined in Definition \ref{defupsilon}.
\end{theorem}


This formula and the maps $\Upsilon :B_{\E_{n+1}}(\N(\Sigma X)) \rightarrow B_{\E_n}(\N(X))$ can be used to define the stable Hopf invariants of a space, which we denote by $B^S(X)$ (see Definition \ref{defBS}). The stable Hopf invariants give us invariants of stable homotopy classes of maps. The pairing from Theorem \ref{theorem1} can then be extended to the following pairing.

\begin{theorem}\label{theorem4}
Let $X$ be a simply-connected space, then there exists a pairing 
\[
\left<,\right>_S:H_*(B^S(X))\otimes \pi_*^S(X) \rightarrow \K.
\]
\end{theorem}

This theorem can be found as Theorem \ref{thrmBS} below. 

As a consequence of Theorem \ref{theorem3} and Theorem \ref{theorem4} we get a criterion which tells us if a map $f:S^k \rightarrow X$ induces a non-trivial element in the stable homotopy groups of $X$.

\begin{corollary}
Let $f:S^k \rightarrow X$ be a map such that $\left<\omega,f\right>_{\E_n} \neq 0$ for a certain $\omega\in B_{\E_n}\N(X)$. If $\omega$ comes from $B^S(X)$, then the class of $f$ in the stable homotopy groups of $X$ is non-trivial. 
\end{corollary}

The corollary can be found as Corollary \ref{corBS} below.





\subsection{Acknowledgements}

The author would like to thank Alexander Berglund and Dev Sinha for many useful conversations about this project. Further the author would like to thank Clemens Berger for  help with the proof of Theorem \ref{thrmsuspensionfunctor}. The author also acknowledges the financial support from Grant  GA CR  No. P201/12/G028.

\part{Preliminaries}


We will assume that the reader is sufficiently familiar with the theory of Koszul duality for operads and algebras over operads, otherwise we refer the reader to the standard reference \cite{LV}. Unless stated otherwise all definitions and notations are taken from \cite{LV} and can be found in this book.   There is only one major difference between this paper and \cite{LV}, and that is that \cite{LV} assumes that they work over a field of characterstic $0$, and that in this paper we will almost exclusively work over a field of characteristic $p >0$. Most of the results of \cite{LV} hold in greater generality than presented there, but there are a few important differences, which we will discus in this section.



\section{Conventions}

In this paper we will always assume that, unless stated otherwise, we work over a field $\K$ of characteristic $p \geq 0$. We denote the symmetric group on $n$ letters by $\S_n$. If $V$ is a vector space with an action of a finite group $G$, then we denote the coinvariants with respect to $G$ by $V_G$ and the invariants with respect to $G$ by $V^G$. In this paper spaces are implicitly assumed to be simplicial sets. We also assume that the sphere $S^n$ will always be sufficiently triangulated such that maps from the sphere to a simplicial set $X$ can be realized as a map of simplicial sets. The reduced suspension of a simplicial set $X$ is denoted by $\Sigma X$ (see Section \ref{subsecreducedsuspensions} for a detailed definition of the reduced suspension), we will always implicitly assume that all suspensions are reduced. By $\NN(X)$ we denote the reduced normalized chains on the simplicial set $X$, dually $\N(X)$ denotes the reduced normalized cochains on $X$.




\subsection{Grading conventions}

In this paper will always work in the category of chain complexes and therefore use a homological grading. This will imply that normalized cochains on a simplicial set will be concentrated in non-positive   degrees. The operations on the normalized cochains will therefore all have non-negative degrees, which implies that almost all the operads  and cooperads considered in this paper will be concentrated in non-negative degrees. 

Let $\K s$ be the one dimensional chain complex generated by an element $s$ of degree $1$. The suspension of a cochain complex $V_*$ is denoted by $s V_*$ and defined by $sV=\K s \otimes V $. More explicitly, this chain complex is defined by  $s V_*=V_{*-1}$, the signs are given by the Koszul sign rule. The linear dual of a chain complex $V$ will be denoted by $V^{\vee}$, note that with our grading conventions this turns all the negative degrees into positive degrees and positive into negative degrees.

Also note that because we always work in the category of chain complexes and grade the cochains non-positively, the cochains on the reduced suspension of a simplicial set $X$ are isomorphic (as chain complexes) to the desuspension of the cochains of $X$, i.e. $\N(\Sigma X) \cong s^{-1}\N(X)$. To distinguish homology from cohomology and chains from cochains we will still denote the cochains and cohomology with a superscript.


\section{Operads, cooperads, algebras, coalgebras and Koszul duality in characteristic $p$}

One of the main differences between working over a field of characteristic $0$ and a field of characteristic $p>0$ is that we no longer have an isomorphism between the invariants and coinvariants with respect to the symmetric group action. In this section we will briefly explain what kind of consequences this has for our algebras and coalgebras. Most of this section will be based on \cite{Fres1} and \cite{Fres2}, but we try to use the notation from \cite{LV}.

\subsection{$\S$-modules, operads and cooperads}

Recall from Chapter 5 of \cite{LV} that an $\S$-module $M_*$ is a sequence 
\[
M_*=(M(0),M(1),...,M(n),...)
\] 
of chain complexes, such that each $M(n)$ is an $\S_n$-module. Recall that an $\S$-module is called connected if $M(0)=0$ and $M(1)=\K$, in the rest of this paper we will always assume that all $\S$-modules are connected, unless stated otherwise. 

On the category of $\S$-modules we can define two products, called the composition product and the composition product with divided symmetries. When we work over a field of characteristic zero these products are isomorphic, but over a field of characteristic $p$ these products are different.

\begin{definition}
Let $M$ and $N$ be two not necessarily connected $\S$-modules, then we define the composition product $M \circ N$ of $M$ and $N$ by
\[
M \circ N= \bigoplus_{n \geq 0} \left( M(n) \otimes N^{\otimes n} \right)_{\S_n}.
\]
The composition product with divided symmetries $M \tilde{\circ} N$ is defined by 
\[
M \tilde{\circ} N= \bigoplus_{n \geq 0} \left( M(n) \otimes N^{\otimes n} \right)^{\S_n}.
\]
\end{definition}

\begin{remark}
Note that we define both products by using a direct sum instead of a product. In the definition of coalgebras sometimes the product is used intead of a  direct sum. In this paper we will assume that all coalgebras are conilpotent and therefore factor through the direct sum. For more details see Chapter 5 of \cite{LV}.
\end{remark}

\begin{definition}
An operad $\mathcal{P}$ is an $\S$-module which is a monoid in the category of $\S$-modules. In particular it is an $\S$-module together with a map $\gamma:\mathcal{P}\circ \mathcal{P}\rightarrow \mathcal{P}$ satisfying certain conditions.

Dually, a cooperad $\mathcal{C}$ is a comonoid in the category of $\S$-modules with the composition product with divided symmetries, i.e. it is an $\S$-module $\mathcal{C}$ with a map  $\Delta:\mathcal{C}\rightarrow \mathcal{C} \tilde{\circ} \mathcal{C}$ satisfying certain conditions. 

An operad $\mathcal{P}$ (resp. cooperad $\mathcal{C}$) is called connected  if the underlying $\S$-module is connected.
\end{definition}

\begin{remark}
Note that our definition of connected differs from \cite{Fres2} and assumes that $\P(0)=0$ and $\P(1)=\K$, instead of just $\P(0)=0$.
\end{remark}

As is explained in Section 1 of  \cite{Fres2} there is a natural transformation from the composition product to the composition product with divided symmetries. This natural transformation comes from the norm map which is a natural map between the coinvariants and invariants of an $\S_n$-module and is defined as follows. 

Let $X$ be an $\S_n$-module then we define a map 
$$Tr:X_{\S_n} \rightarrow X^{\S_n}$$
$$Tr(x)=\sum_{\sigma \in \S_n} \sigma x.$$

Let $M$ and $N$ be $\S$-modules, then the norm map extends to a natural transformation from the composition product $M \circ N$ to the composition product  with divided symmetries $M \tilde{\circ} N$. This map is denoted by

$$Tr_{M,N}:M \circ N \rightarrow M \tilde{\circ} N,$$
and defined by applying the norm map level-wise.

The norm map is in general not an isomorphism, but as is shown in Proposition 1.1.15 in \cite{Fres2}, under some assumptions on $N$ it is an isomorphism.  



\begin{proposition}[\cite{Fres2}, Proposition 1.1.15]
Let $M$ and $N$ be $\S$-modules such that $N$ is connected, then the norm map
$$Tr_{M,N}:M \circ N \rightarrow M \tilde{\circ}N$$
 is an isomorphism.
\end{proposition}

Recall that we assumed that all our operads and cooperads are connected. So in particular if $\mathcal{C}$ is a cooperad, the norm map induces an isomorphism between $\mathcal{C}\circ \mathcal{C}$ and $\mathcal{C} \tilde{\circ}\mathcal{C}$. Because of this isomorphism we will from now on assume that a cooperad is an $\S$-module $\mathcal{C}$ with a map $\Delta:\mathcal{C} \rightarrow \mathcal{C} \circ \mathcal{C}$ satisfying certain conditions. So from now on we will always assume that the structure map of a cooperad goes to the coinvariants instead of the invariants, and otherwise we will implicitly apply the inverse of the norm map to land in the coinvariants. 

With this convention, the theory of Koszul duality for operads and cooperads works exactly the same as in \cite{LV}. We do not repeat this theory here and assume that the reader is familiar with \cite{LV}, whose definitions, notations and conventions we will use. For more details about Koszul duality for operads in characteristic $p$ we also refer the reader to \cite{Fres1}, in which the proofs also give  explicit details about  the characteristic $p$ case.

\subsection{Algebras and coalgebras}

The main difference between working over a field of characteristic $0$ and characteristic $p>0$ is that there are now two possible definitions of algebras and coalgebras, one using the composition product and one using the composition product with divided symmetries.

\begin{definition}
Let $\mathcal{P}$ be an operad. An algebra over $\mathcal{P}$ is a chain complex $A$, together with a map $\mu_A:\mathcal{P} \circ A \rightarrow A$, which is associative and unital with respect to the operad composition maps. 

A $\mathcal{P}$-algebra with divided symmetries $A$ is a chain complex $A$, together with a map $\mu_A:\mathcal{P} \tilde{\circ}A \rightarrow A$, which is associative and unital with respect to the operadic composition maps.
\end{definition}

Note that we implicitly interpet $A$ as an $\S$-module by putting $A$ in arity $0$. For the precise definition see \cite{Fres2}. Dually to the algebra case we also have two possible definitions for coalgebras over a cooperad, which we  give  in the next definition.

\begin{definition}\label{defcoalgebras}
Let $\mathcal{C}$ be a cooperad. A $\mathcal{C}$-coalgebra is a chain complex $C$, together with a map $\Delta_C:C \rightarrow \mathcal{C} \tilde{\circ} C$, which is coassociative and counital with respect to the cooperadic decomposition maps.

A $\mathcal{C}$-coalgebra with divided symmetries is a chain complex $C$, together with a map $\Delta_C:C \rightarrow \mathcal{C} \circ C$, which is coassociative and counital with respect to the cooperadic decomposition maps.
\end{definition}

\begin{remark}
Because we defined  the composition product with a direct sum instead of a product, the coalgebras (with or without divided symmetries) defined in Definition \ref{defcoalgebras}, should technically be called conilpotent coalgebras (with or without divided symmetries). In this paper we will always assume that all the coalgebras (with or without dived symmetries) are conilpotent. See Section 5.8.5 of \cite{LV} for more details.

\end{remark}

As is explained in Section 1.1.20 of \cite{Fres2}, the norm map also gives us a relation between algebras with divided symmetries and algebras without divided symmetries. In particular the norm map gives us a map 
\[
Tr_{\P,A}:\P \circ A \rightarrow \P \tilde{\circ} A,
\] 
So every $\P$-algebra with divided symmetries is also a $\P$-algebra without divided symmetries. Dually every $\C$-coalgebra $C$ with divided symmetries is also a $\C$-coalgebra without divided symmetries. The structure map is given by the structure map of $C$ composed with the norm map, explicitly it is given by
\[
C \xrightarrow{\Delta} \C \circ C \xrightarrow{Tr_{\C,C}} \C \tilde{\circ} C.
\]


For some examples of algebras with divided symmetries we refer the reader to \cite{Fres2}.



\subsection{Operadic suspension}

In the rest of the paper and in particular the application part the operadic suspension will play an important role, we will here recall these definitions, for more details see Section 7.2.2 of \cite{LV}. 

Let $\mathcal{S}$ be the endomorphism operad of $\K s$, the one dimensional chain complex concentrated in degree $1$, i.e. $\mathcal{S}(n)=Hom_{\K}(\K s ^{\otimes n},\K s)$. The operadic suspension of an operad $\mathcal{P}$ is then defined as the Hadamard or aritywise tensor product with $\mathcal{S}$.

The operadic suspension has the property that if $A$ is a $\mathcal{P}$-algebra, then $sA$ is an $\mathcal{SP}$-algebra.

Dually, the cooperadic suspension of a cooperad $\mathcal{C}$ is defined by taking the aritywise tensor product of $\mathcal{C}$ with $\mathcal{S}^c$, where $\mathcal{S}^c$ is the coendomorphism cooperad of $\K s$, i.e. $\mathcal{S}^c(n)=Hom_{\K}(\K s , \K s^{\otimes n})$. Similar to operads, the suspension of a coalgebra $C$ over a cooperad $\mathcal{C}$ is a coalgebra over $\mathcal{S}^c \mathcal{C}$.

\section{The bar construction relative to an operadic twisting morphism}\label{secthebarconstruction}

For the definition of the $\E_n$-Hopf invaraints we will need the bar construction relative to an operadic twisting morphism. We will now adapt the definition of the relative bar construction from Chapter 11 of \cite{LV}, to work over a field of characteristic $p$. Let $\alpha:\mathcal{C}\rightarrow \mathcal{P}$ be an operadic twisting morphism between a cooperad $\mathcal{C}$ and an operad $\mathcal{P}$. In \cite{LV}, they define the bar and cobar construction as a pair of adjoint functors between the category of $\mathcal{P}$-algebras and conilpotent $\mathcal{C}$-coalgebras.  The main difference is that the bar-cobar adjunction in characteristic $p$, is no longer an adjunction between the category of $\mathcal{P}$-algebras and the category of $\mathcal{C}$-coalgebras, but it is in the characteristic $p$ case an adjunction between the category of $\mathcal{P}$-algebras and the category of  $\mathcal{C}$-coalgebras with divided symmetries. This additional structure will not play an important role in this paper, but for the sake of preciseness and completeness we give a detailed description of the bar construction here.


\begin{definition}
Let $\alpha:\mathcal{C} \rightarrow \mathcal{P}$ be an operadic twisting morphism between a cooperad $\mathcal{C}$ and an operad $\mathcal{P}$. Let $C$ be a $\mathcal{C}$-coalgebra with divided symmetries and $A$ a $\mathcal{P}$-algebra.

A linear map $\phi:C \rightarrow A$ of degree $0$ is called a twisting morphism if it satisfies the Maurer-Cartan equation, which is given by 
\[
\partial(\phi) + \star_{\alpha}(\phi)=0,
\]
where $\star_{\alpha}$ is the operator defined by
\[
\star_{\alpha}:C \xrightarrow{\Delta_C} \mathcal{C} \circ C \xrightarrow{\alpha \circ \phi} \mathcal{P} \circ A \xrightarrow{\gamma_A} A,
\]
where $\Delta_C$ is the coproduct of $C$ and $\gamma_A$ is the product of $A$. The set of twisting morphisms relative to $\alpha$ is denoted by $Tw_{\alpha}(C,A)\subset Hom_{\K}(C,A)$.
\end{definition}

Notice that this definition would not work if $C$ is a coalgebra without divided symmetries, in this  case it would not be possible to define the $\star_{\alpha}$-operator, since the source and target of the middle map would involve different symmetries.

Similarly to \cite{LV}, we can now define the notion of bar and cobar construction. The bar and cobar construction will represent the set of twisting morphisms  and when the operadic twisting morphism $\alpha:\mathcal{C} \rightarrow \mathcal{P}$ is Koszul it will also give us a method to construct functorial free resolutions.



\begin{definition}
Let $\alpha:\mathcal{C} \rightarrow \mathcal{P}$ be an operadic twisting morphism. Then we define a functor 
\[
B_{\alpha}:\mathcal{P}\mbox{-algebras} \longrightarrow \mbox{Conilpotent } \mathcal{C}\mbox{-coalgebras with divided symmetries},
\]
from the category of $\mathcal{P}$-algebras to the category of conilpotent $\mathcal{C}$-coalgebras with divided symmetries. This functor is defined as follows. Let $A$  be a $\mathcal{P}$-algebra, then we define $B_{\alpha}A$ as $(\mathcal{C}(A),d_B)$, the conilpotent quasi-cofree  $\mathcal{C}$-coalgebra with divided symmetries cogenerated by $A$. The differential $d_B$ is defined as $d_B=d_1+d_2+d_3$, where $d_1$ is the extension of $d_A$ to $\mathcal{C}(A)$,which is given by $Id_{\C} \circ'd_{A}$ (see Section 6.1.3 of \cite{LV} for the precise definition of $\circ'$). The differential $d_2$ is the unique coderivation which extends the degree $1$ map 

\[
\mathcal{C} \circ A \xrightarrow{\alpha \circ Id_A} \mathcal{P} \circ A \xrightarrow{\gamma_A} A, 
\]

where $\gamma_A$ is the structure map of $A$. More explicitly, the coderivation $d_2$ is given by the following sequence of maps:

\[
\mathcal{C} \circ A \xrightarrow{\Delta_{(1)}\circ Id_A} (\mathcal{C} \circ_{(1)} \mathcal{C}) \circ A \xrightarrow{(Id_{\mathcal{C}} \circ_{(1)}\alpha) \circ Id_A} (\mathcal{C} \circ_{(1)}\mathcal{P}) \circ A
\]
\[
 \cong \mathcal{C} \circ (A; \mathcal{P} \circ A) \xrightarrow{Id_{\mathcal{C}} \circ (Id_A; \gamma_A)} \mathcal{C} \circ A,
\]
where $\circ_{(1)}$ denotes the infinitesimal composition product (see Section 6.1 of \cite{LV}). The differential $d_3$ is given as the unique extension of the differential $d_{\C}$, explicitly it is given by $d_{\C} \circ Id_A$.
\end{definition}
 
 
The following lemma will be important in Section \ref{secEnhopfinv}.

\begin{lemma}\label{lemmadifferentialsanticommute}
Let $\alpha:\mathcal{C} \rightarrow \mathcal{P}$ be an operadic twisting morphism and let $A$ be a $\mathcal{P}$-algebra. Then the bar differential $d_{B}$ squares to zero, furthermore we have the following identities $d_1^2=0$, $d_2^2=0$, $d_3^2 =0$, $d_1d_2+d_2d_1=0$, $d_1d_3+d_3d_1=0$ and $d_2d_3+d_3d_2=0$.
\end{lemma} 
 
A large part of the proof of this lemma can be found as part of Section 11.2.2 of \cite{LV}, the rest is straightforward and left to the reader.
 
In the rest of this paper we will need an additional grading on the bar construction  called the weight grading. This grading will be important in Sections \ref{secEnhopfinv} for the definition of the $\E_n$-Hopf invariants.

\begin{definition}
Let $\alpha:\mathcal{C} \rightarrow \mathcal{P}$ be a Koszul twisting morphism and let $A$ be a $\mathcal{P}$-algebra. We define an additional grading on the bar construction by defining the elements of weight $n$ as the elements of  $\mathcal{C}(n)\otimes_{\S_n}A^{\otimes n} \subset B_{\alpha}A$.
\end{definition}

The following property of the bar construction will be very important for the constructions in this paper. This proposition is a generalization of Proposition 11.2.3 of \cite{LV} and Lemma 2.2.5 of \cite{Mil1}.

\begin{proposition}\label{propbarconstructionpreservesqi}
Let $\alpha:\mathcal{C} \rightarrow \mathcal{P}$ be an operadic twisting morphism between a cooperad $\mathcal{C}$ and an operad $\mathcal{P}$. Further assume that $\mathcal{C}(n)_d$ and $\mathcal{P}(n)_d$ are projective $\S_n$-modules for all $d$ and $n$ and that $\mathcal{C}$ and $\mathcal{P}$ are concentrated in non-negative degrees. Let $f:A \rightarrow A'$ be a quasi-isomorphism, then the induced morphism $B_{\alpha}f:B_{\alpha}A \rightarrow B_{\alpha}A'$ is a quasi-isomorphism as well.
\end{proposition}


\begin{proof}
To prove the proposition we will use a spectral sequence argument similar to the proof of Proposition 11.2.3 of \cite{LV}. First we define a filtration on  $B_{\alpha}A$ and $B_{\alpha}A'$ by defining $F_pB_{\alpha}A=\mathcal{C}_{\leq p}(A)$, where $\mathcal{C}_{\leq p}$ is the subcooperad of $\mathcal{C}$ consisting of elements of homological degree less or equal than $p$.  It is easy to check that this is indeed a filtration and that the filtration is stable under the bar differential. Recall that the bar differential splits as $d_{B}=d_1+d_2+d_3$, where $d_1$ is the differential coming from $d_A$, the differential of $A$, $d_2$ is the differential coming from the multiplication of $A$ and $d_3$ is the differential coming from $d_{\C}$, the differential of $\mathcal{C}$. Since $\mathcal{C}$ and $\mathcal{P}$ are non-negatively graded, it is easy to check that the differentials respect the filtration in the following way,  $d_1:F_pB_{\alpha}A \rightarrow F_pB_{\alpha}A$, $d_2:F_pB_{\alpha}A \rightarrow F_{\leq p-1}B_{\alpha}A$ and $d_3:F_pB_{\alpha}A \rightarrow F_{p-1}B_{\alpha}A$. The $E^0$-page of the corresponding spectral sequence is then given by $E^0_{p,q}B_{\alpha}A\cong(\mathcal{C}_p(A))_{p+q}$ and the differential is given by $d_1$. Since we assumed that all coalgebras are conilpotent, this filtration is increasing and bounded below, the spectral sequence therefore converges.

Since we assumed that $\mathcal{C}(n)_p$ is projective as an $\S_n$-module, we can now apply the K\"unneth Theorem (see \cite{MacLaneH} Theorem V.10.1), to compute the $E^1$-page of this spectral sequence. The K\"unneth Theorem states that if $\mathcal{C}$ is projective as an $\S_n$-module, then $H_*(\mathcal{C}_d(A))$ is isomorphic to $H_*(\oplus_{n\geq 1} \mathcal{C}(n)_p \otimes_{\S_n} A^{\otimes n})$ which is by the K\"unneth formula isomorphic to $\oplus_{n\geq 1} \mathcal{C}(n)_p \otimes_{\S_n} H_*(A^{\otimes n})$. Since we are working over a field, $H_*(A^{\otimes n})$ is isomorphic to $H_*(A)^{\otimes n}$. The $E^1$-page therefore becomes $E^1_{p,q}B_{\alpha}A=(\mathcal{C}_p(H_*(A))_{p+q}$. Since $A$ and $A'$ are quasi-isomorphic, the induced map $E^1f$ is now an isomorphism between $E^1B_{\alpha}A$ and $E^1B_{\alpha}A'$. Because both spectral sequences converge we can apply the Classical Convergence Theorem 5.5.1 of \cite{Wei1}, which implies that the map $B_{\alpha}f:B_{\alpha}A \rightarrow B_{\alpha}A'$ is a quasi-isomorphism. This proves the proposition.
\end{proof}

\begin{remark}
Similar to the bar construction we can also define the cobar construction. The cobar construction will not be needed for the results of this paper and we will therefore not explicitly describe it here. Similarly it will also be possible to define the bar and cobar construction for algebras with divided symmetries, this will also not be relevant for this paper and we therefore choose to ignore it.

\end{remark}









\section{$\E_n$- and $\E_{\infty}$-algebras}\label{secEnalgebras}

In this section we recall the definitions and basic facts about the theory of $\E_n$- and $\E_{\infty}$-algebras which we use in this paper. In general an $\E_{\infty}$-operad is an operad that is weakly equivalent to $\mathcal{COM}$, the operad encoding commutative algebras. In this paper we will use two specific $\E_{\infty}$-operads  called the Barratt-Eccles operad and the surjection operad. The Barratt-Eccles operad has several technical advantages coming from the fact that it comes from an operad in simplicial sets. It also has the advantage that the Koszul duality for the Barratt-Eccles operad has been described in great detail in \cite{Fres3}. The advantage of the surjection operad is that it is a bit smaller than the Barratt-Eccles operad and that the action of the surjection operad on the normalized cochains of a simplicial set is easier to describe. Most of this section is based on \cite{Fres3} and \cite{BF1}, see also \cite{MS4}.





\begin{remark}
Since we are using a homological grading, the cochains of a space will be concentrated in negative degrees, we will however still call them the  cochains on that space.
\end{remark}


\subsection{The Barratt-Eccles operad}

We will ow describe the Barratt-Eccles operad. This $\E_{\infty}$-operad has several technical advantages, the first advantage is that the Barratt-Eccles operad is combinatorial and that there is a clear description on how it acts on the cochains of a space. A second advantage is that it comes equipped with a filtration by suboperads callled $\E_n$-operads. The last advantage is that it is a Hopf operad, this allows us to take tensor products and has the consequence that the category of algebras over the Barratt-Eccles operad forms a model category. Most of this section is based on \cite{BF1}. For the applications of the $\E_n$-Hopf invariants we describe in this paper, the composition map will  not be important, we will therefore omit  a description of the composition map and refer the reader to \cite{BF1} for a detailed description of this map. 

 
\begin{definition}
 The Barratt-Eccles operad $\mathcal{E}_\infty$, is the operad given by the following $\S$-module, $\mathcal{E}_\infty(r)_d$, the arity $r$ degree $d$ part of the Barratt-Eccles operad,  is generated all by $(d+1)$-tuples of the form $(w_0,...,w_{d})$, where $w_i \in  \S_r$. We further require that if $w_i=w_{i+1}$ for some $i$, then the element is set to zero.  The action of an element $\sigma \in S_r$ on a tuple $(w_0,...,w_{d})$ is given by $\sigma (w_0,...,w_{d})= ( \sigma (w_0),...,\sigma (w_d))$. The differential is given by the classical formula by omitting elements, i.e.
\[ 
 d(w_0,...,w_d)=\sum_{i=0}^d (-1)^i(w_0,...,\hat{w}_i,...,w_d),
 \]
 where $\hat{w}_i$ means that we omit the $i$th element.
\end{definition}


One of the technical advantages of the Barratt-Eccles operad is that its algebras  form a model category (for the basics on model categories see for example \cite{DS1}).

\begin{theorem}[\cite{BF1}, Theorem 3.1.1]
The category of algebras over the Barratt-Eccles operad forms a model category in which the weak-equivalences are given by quasi-isomorphisms of the underlying chain complexes and the fibrations are the morphisms whose maps of underlying chain complexes are degree-wise surjective.
\end{theorem}

As an example we will now describe the arity $2$ part of the Barratt-Eccles operad. Denote by $Id$ and $\tau$ the two elements of $\S_2$. In degree $d$, the space $\mathcal{E}(2)_d$ is two dimensional and given by the sequences $(Id,\tau,Id,\tau,...)$ and $(\tau,Id,\tau,Id,...)$, where both sequences have length $d+1$.

The Barratt-Eccles operad comes with a filtration by a sequence of suboperads called the $\mathcal{E}_n$-operads. Topologically, this filtration  corresponds to the little $n$-cubes filtration of the topological $\E_{\infty}$-operad. The main reason we need these $\mathcal{E}_n$-suboperads is because the sphere $S^n$ is not formal as an $E_{\infty}$-algebra, but the sphere $S^n$ is formal as an $\mathcal{E}_m$-algebra for $m<n$.

To define the filtration of the Barratt-Eccles operad we will assign a natural number to each element of the Barratt-Eccles operad, this number is  the number of variations in the element. The $\mathcal{E}_n$-operad is then defined as  the suboperad consisting of all elements of $\E_\infty$ which have less than $n$ variations.

Let $(w_0,...,w_d)\in \mathcal{E}(r)_d$ with $w_i\in \S_r$, then we can represent every permutation $w_i$ by the sequence of its values. Denote by $w_i \vert_{j,k}$ the permutation of $j$ and $k$ defined by restricting the sequence to $j$ and $k$. 

We define the number of variations $\mu_{j,k}$ of $j$ and $k$ for an element $(w_0,...,w_d)$ as the number of times that $w_i\vert_{j,k} \neq w_{i+1}\vert_{j,k}$ in the sequence $(w_0\vert_{j,k},...,w_d\vert_{j,k})$. The $\mathcal{E}_n$-operad is defined as the subspace of $\mathcal{E}_\infty$ spanned by all elements such that $\mu_{j,k} <n$ for all $j$ and $k$.

 As is shown in \cite{Fres3}, the $\mathcal{E}_1$-operad is the associative operad.




For more details and examples see Section 0.1.3 of \cite{Fres3}  or Section 1.6 \cite{BF1}. The following lemma will be important in Section \ref{secEnhopfinv} and can be found as Lemma 3.3.4 in \cite{Fres3}.

\begin{lemma}\label{lemenisbounded}
The $\mathcal{E}_n$-operad is finite dimensional in every arity, i.e. $\mathcal{E}_n(r)$ is a finite dimensional chain complex. In particular the underlying chain complex is concentrated in   degrees greater or equal than $0$ and less than $nr(r-1)/2$. In particular, the arity two component $\E_n(2)_d$ of the $\E_n$-operad is only non-zero when $0 \leq d \leq n-1$.
\end{lemma}




\subsection{The surjection operad}

In this section we will recall the definition of the surjection operad $\mathcal{X}_\infty$.  Let $u:\{1,...,r+d\} \rightarrow \{1,...,r\}$ be a surjective map, then we can describe $u$ by its image $(u(1),...,u(r+d))$, which is a sequence of the numbers $\{1,...,r\}$ of length $r+d$, such that each number in the set $\{1,..,r\}$ appears at least once. We will call a surjection $u$ non-degenerate if there are no repetitions in this sequence, i.e. if $u(n) \neq u(n+1)$ for all $n\in \{1,...,r+d-1\}$.

As an $\S$-module, the arity $r$ degree $d$  part of the surjection operad $\mathcal{X}_\infty(r)_{d}$ is spanned by all non-degenerate surjections $u:\{1,...,r+d\} \rightarrow \{1,...,r\}$. The action of $\S_r$ is given by permuting the image of $u$, i.e. $\sigma u$ is the sequence $(\sigma (u(1)),...,\sigma (u(r+d)))$. The differential is given by 
\[
\partial (u(1),...,u(r+d))= \sum_{i=1}^{r+d}\pm(u(1),..., \widehat{u(i)},...,u(r+d)),
\]
where $\widehat{u(i)}$ means that we omit the $i$th element in the sequence. For the definition of the sign see Section 1.2.3 of \cite{BF1}. If the element $u(i)$ only appears once in the sequence $u$, then we set the term where we omit $u(i)$ in the differential to zero.  We will not describe the composition maps in this paper since they will not be relevant for the applications. 

Similar to the Barratt-Eccles operad the surjection operad, also has a filtration which is defined as follows. For each surjection $u:\{1,...,r+d\} \rightarrow \{1,...,r\}$ and numbers $i<j$ we define  $u_{i,j}$ as the subsequence of $u$ formed by the occurrences of $i$ and $j$. Let $\mu_{i,j}$ be the number of variations in the sequence $u_{i,j}$. 

The  filtration $\X_n$ of $\X_\infty$ is defined as the subspace of $\X_\infty$ consisting of all surjections $u$, such that $u_{i,j}$ has less than $n$ variations for all $i$ and $j$.

In \cite{BF1}, Berger and Fresse define a morphism of operads from the Barratt-Eccles operad to the surjection operad called the table reduction morphism. The properties of this morphism are stated in the following lemma.

\begin{lemma}[\cite{BF1}, Lemma 1.6.1]
The table reduction morphism $TR:\E_{\infty} \rightarrow \X_\infty$ is a surjective quasi-isomorphism of operads, it preserves the filtrations and induces  quasi-isomorphisms $TR_n:\E_n \rightarrow \X_n$ for all $n>0$. 
\end{lemma}



\subsection{The normalized cochains on a simplicial set}\label{seccochainEinftyalg}

The main reason we are interested in $\E_{\infty}$-algebras is because the singular chains on a topological space and the reduced normalized chains on a simplicial set form an $\mathcal{E}_\infty$-coalgebra and dually the cochains on a space form an $\E_\infty$-algebra.


\begin{theorem}[\cite{BF1}, Theorem 2.1.1]\label{thrmBergerFresse}
 Let $X$ be a simplicial set, then the reduced normalized cochain complex $\N(X)$ is an algebra over the surjection operad $\X_\infty$. Because of the table reduction morphism $\N(X)$ will also naturally be an algebra over the Barratt-Eccles operad $\mathcal{E}_\infty$.
\end{theorem}


We will now briefly recall how the action of the surjection operad on the normalized cochains is defined. We will be fairly brief and more details can be found in \cite{BF1} (see also \cite{MS4}).

According to Section 2.2 of \cite{BF1}, the normalized chains on a simplicial set $X$ form a coalgebra over the surjection operad (see \cite{BF1} Section 2.1.2 for the definition of a coalgebra over an operad). We will now recall this coalgebra structure and explain how it induces an algebra structure on the normalized cochains.  To describe this coalgebra structure we first recall the definition of the interval cut operation.

To define the interval cut operation we  first define it on the standard simplex and then we extend it to general simplices in a general simplicial set by taking the image of the cut simplices. Let $\Delta^n$ be the $n$-dimensional simplex and let $u:\{1,...,r+d\} \rightarrow \{1,...,r\}$ be a non-degenerate surjection. We define for $k \in \{1,..,r\}$ a set of subsimplices $\Delta(C_{(k)})$ of $\Delta^n$ as follows.

First we cut the interval $[0,...,n]$ in $r+d$ pieces, such that two consecutive pieces overlap in exactly one point. The pieces of this cut are specified by the end points of the interval and the $(r+d-1)$-points where we cut the interval. Note that the pieces are allowed to have length zero. So we get a sequence of integers
\[
0 =n_0 \leq n_1 \leq ... \leq n_{r+d-1} \leq n_{r+d}=n.
\]
The set $C_{(k)}$ is then defined as the union of the intervals $[n_{i-1},n_i]$ for all $i$ such that $u(i)=k$. The simplex $\Delta(C_{(k)})$ is then defined as the subsimplex of $\Delta^n$ which has vertices given by the set $C_{(k)}$. The cooperation associated to $u$ is then given by 
\[
u:\NN(\Delta^n) \rightarrow \NN(\Delta^n)^{\otimes r}
\] 
\[
u(\Delta^n)=\sum \pm\Delta(C_{(1)}) \otimes .... \otimes \Delta(C_{(r)}),
\]
where the sum ranges over all possible ways to cut the interval $[0,...,n]$ in $r+d$ pieces. Note that whenever a set $C_{(k)}$ is degenerate, i.e. some element appears twice in $C_{(k)}$, we will set the corresponding simplex to zero in $\NN(\Delta^n)$.

We extend this coproduct to a general simplicial set as follows. Let $X$ be a simplicial set and $x:\Delta^n \rightarrow X$ be a non-degenerate simplex. Then we define 
\[
u:\NN(X)\rightarrow \NN(X)^{\otimes r}
\]
by
\[
u(x)=\sum \pm x(\Delta(C_{(1)})) \otimes ... \otimes x(\Delta(C_{(r)})),
\]
where $x(\Delta(C_{(k)}))$ is the image of the restriction of $x$ applied to the subsimplex $\Delta(C_{(k)})\subseteq \Delta^n$.




Using this coalgebra structure on the normalized chains, we can define an $\X_\infty$-algebra structure on the normalized cochains as follows. Let $\phi_1,...,\phi_r \in \N(X)$ and $u:\{1,...,r+d\} \rightarrow \{1,...,r\}$ be a non-degenerate surjection. Then we define an operation of degree $d$
\[
u:\N(X)^{\otimes r} \rightarrow \N(X),
\]
by evaluating the $\phi_i$ on the coproduct of a simplex $x$. More precisely if $x:\Delta^n\rightarrow X$ is a non-degenerate simplex, then the evaluation of $u(\phi_1,...,\phi_r)$ on $x$ is defined by
\[
u(\phi_1,..,\phi_r)(x)=\sum \pm \phi_1 \otimes ... \otimes \phi_r(x(\Delta(C_{(1)})) \otimes ... \otimes x(\Delta(C_{(r)}))),
\]
where the sum runs again over all possible ways to cut the interval $[0,...,n]$ in $r+d$ pieces. Note that when the degree of $\phi_k$ does not match the dimension of $x(\Delta(C_{(k)}))$, this evaluation is equal to zero. By abuse of notation we will denote both the operation and cooperation associated to $u$ by $u$. From the context it should be clear which one we mean.

\subsubsection{Signs associated to the interval cut operation}

In Section \ref{secsuspensions}, we need explicit formulas for the signs associated to the coproduct associated to a surjection $u \in \X_{\infty}(r)_d)$. We recall here from Section 2.2.4 from \cite{BF1} how to associate a sign to the operation
\[
u(\Delta^n)=\sum \pm\Delta(C_{(1)}) \otimes .... \otimes \Delta(C_{(r)}).
\]
The sign $\pm$ consists of two parts, the first part is called the position sign and the second part the permutation sign.  

To describe these signs let $0\leq n_1 \leq ... \leq n_{r+d-1}\leq n$ be the points where the interval is cut. To these values we associate $r+d$ intervals $[n_{i-1},n_i]$. There are two different types of intervals, an interval $[n_{i-1},n_i]$ is called a final interval if $u(i)$ is the last occurrence of the value $u(i)$ in the sequence $(u(1),...,u(r+d))$ and all the non-final intervals are called inner intervals.

The length of an inner interval $[n_{i-1},n_i]$ is defined as $n_i-n_{i-1}+1$ and the length of an outer interval $[n_{i-1},n_i]$ is defined as $n_i-n_{i-1}$. The permutation sign is then defined as follows.  Consider the shuffle that permutes the sequence $(u(1),u(2),...,u(r+d))$ to the sequence $(1,..,1,2,...,2,...,r,...,r)$. The permutation sign is then defined as the sign produced by the Koszul sign rule by interchanging the intervals according to this shuffle, where the degree of an interval is defined as its length.

The position sign is defined as the sum of the end points of the inner intervals. So if the interval $[n_{i-1} ,n_i]$  is inner then it contributes a $(-1)^{n_i}$ to the position sign. The total sign is defined as the product of the permutation sign and the position sign.

\subsection{Spheres}

One of the reasons why $p$-adic homotopy theory is much more difficult than rational homotopy theory, is because in $p$-adic homotopy theory the one-dimensional model for the cohomology of the $m$-dimensional sphere is not a one-dimensional chain complex concentrated in degree $-m$ with the trivial $\mathcal{E}_\infty$-algebra structure, but has some non-trivial higher operations. One well known example of such a non-zero operation is the zeroth Steenrod square, which is always equal to the identity. 

These non-zero operations have the consequence that $S^m$ is not formal as an $\mathcal{E}_\infty$-algebra, but is only formal as an $\mathcal{E}_n$-algebra with $n<m$.  In this paper we avoid the problem of the non-trivial $\E_{\infty}$-algebra structure on $S^m$, by only considering $S^m$ as an $\E_n$-algebra with $n<m$.

In Section 3.2 of \cite{BF1}, Berger and Fresse give explicit formulas for the $\E_{\infty}$-algebra structure on the sphere $S^m$. Since we do not need the explicit formulas in paper we will only state a simplified version of their results.


\begin{theorem}[\cite{BF1}, Proposition 3.2.5]\label{thrmspheresaresomethimesformal}
 The one dimensional model for the sphere $S^m$ is non-trivial as an $\E_{\infty}$-algebra, but is trivial as an $\mathcal{E}_n$-algebra for $n <m$.
\end{theorem}

\begin{remark}
Another nice interpretation of the $\E_\infty$-algebra structure of $\N(S^m)$ can be found in the epilogue of \cite{Fres3}.
\end{remark}

\subsection{Koszul duality for $\mathcal{E}_n$-operads: Part 1}\label{subsecKoszulduality1}

In this section we will recall some of the results about the self Koszul duality of the $\E_n$-operads, in Section \ref{subsecKoszulduality2} we go into these results in more detail. The main result of \cite{Fres3} is that, up to a shift, the $\E_n$-operad is self Koszul dual. This will be an important result we use in this paper.

\begin{theorem}[\cite{Fres3}, Theorem  B]
 There exist quasi-isomorphisms $\psi_n:\Omega (\D_n)\rightarrow \E_n$, where $\D_n=\Sus^{-n}\E_n^{\vee}$, where $\Omega(\D_n)$ denotes the operadic cobar construction on $\D_n$.
\end{theorem}

An immediate corollary of this theorem is the existence of a Koszul twisting morphism from $\D_n$ to $\E_n$.

\begin{corollary}
 For each $n \geq 1$, there exists a Koszul twisting morphism $\tau_n:\D_n \rightarrow \E_n$.
\end{corollary}

The problem with this twisting morphism is that there are no explicit formulas known for the full twisting morphism. Fortunately we do have formulas for the arity two component of this twisting morphism,  which for all the applications and examples in this paper this will be sufficient. Since we do not need a description of the arity two component of $\tau_n$ yet, we postpone it to Section \ref{secexamples}.


From the bound for the $\E_n$-operad from Lemma \ref{lemenisbounded}, we can easily deduce that the $\D_n$-operad is also bounded.

\begin{lemma}\label{lemDnisbounded}
The $\D_n$-cooperad is finite dimensional in each arity, more precisely the chain complex $\D_n(r)$ is concentrated in degrees greater than $\left( -\frac{nr}{2}+n\right)\left(r-1\right)$ up and till $-n(r-1)$. The arity two component $\D_n(2)_d$ is only non-zero for $1\leq d \leq n$.
\end{lemma}

The proof follows straightforwardly from the bound of the $\E_n$-operad from Lemma \ref{lemenisbounded} and is  left to the reader. Using the twisting morphism $\tau_n$ we can define the $\E_n$-bar construction.

\begin{definition}
The bar construction on an $\E_n$-algebra $A$, with respect to the twisting morphism $\tau_n:\D_n \rightarrow \E_n$ will be denoted by $B_{\E_n}A$ and is called the $\E_n$-bar construction.
\end{definition}

\part{$\E_n$-Hopf invariants}

Using the theory described in the first part we can now define a generalization of the classical Hopf invariant using $\E_n$-algebras. For simplicity we will work over a field $\K$ of characteristic $p$ in this part, but many results also work over a more general ring.



\section{The classical Hopf invariant}\label{secclassicalhopfinvarainaiianiteihskbndsbcoisdebioisgksdaaaaaaaaaaaaaaaaaarrrrrrrgggggggg}

One way of defining the classical Hopf invariant is via the associative bar construction. Before we give the definition of the $\E_n$-Hopf invariants, we  first recall  the definition of the Hopf invariants coming from the associative bar construction. This section is based on Section 1 of \cite{SW2}.

Let $f:S^m \rightarrow X$ be a map from the $m$-dimensional sphere to a simply-connected space $X$. We would like to construct invariants of the homotopy class of this map. To do this we will first take the reduced normalized cochains on the spaces $S^m$ and $X$ to get an induced map $f^*:\N(X) \rightarrow \N(S^m)$.

\begin{remark}
For simplicity we are working here with the reduced normalized (co)chains on $S^m$ and $X$, the constructions also work when we take singular cochains or when $\K=\Q$ the polynomial de Rham forms.
\end{remark} 

By Theorem \ref{thrmBergerFresse}, the cochain complexes $\N(X)$ and $\N(S^m)$ are both $\E_{\infty}$-algebras, so in particular they are associative algebras. We can now take the associative bar construction of the map $f^*$ to get a map 
\[
B_{Ass}f^*:B_{Ass}\N(X)\rightarrow B_{Ass}\N(S^m).
\]

Since the sphere $S^m$ is a simple space from a cohomological point of view, we can extend the canonical twisting morphism $\pi:B_{Ass}(\N(S^m)) \rightarrow \N(S^m)$ to a well defined map in cohomology $\pi^*:\H^*(B_{Ass}(\N(S^m)))\rightarrow \H^*(S^m)$. This was done in Section 1 of \cite{SW2}.

\begin{lemma}[\cite{SW2}, Lemma 1.2]\label{lemcohomologyofthebarconstruction}
The $-m$th cohomology group of $B_{Ass}\N(S^m)$ is one-dimensional and generated by $\H^{-m}(S^m)$.
\end{lemma}

The idea of the proof is that for every cocycle $\omega$ of degree $-m$ in weight $w$, we can always find a cocycle $\tau(\omega)$ of weight $1$, such that $\omega$ and   $\tau(\omega)$ are cohomologous in the associative bar construction. For a detailed algorithm for this construction see Section 1 of \cite{SW2}, their algorithm will also be a special case of the algorithm described in Section \ref{secEnhopfinv}.

Using this algorithm, we can construct a map $\pi^*:B_{Ass}(\N(S^m))_{-m}\rightarrow \H^{-m}(S^m)$. So whenever we have a cocycle $\omega \in B_{Ass}(\N(X))_{-m}$, we can pull this back along the map induced by $f:S^m \rightarrow X$ to get a cocycle $B_{Ass}f^*\omega$ in $B_{Ass}(\N(S^m))_{-m}$. We then find a weight one representative of $B_{Ass}f^*\omega$, which we denote by $\tau(B_{Ass}f^* \omega)$. Once we evaluate the cocycle $\tau(B_{Ass}f^*\omega)$ on the fundamental class of  $S^m$ we get a well defined number which is an invariant of the homotopy class of $f$. 

\begin{definition}
Let $X$ be a simply-connected space then we define the Hopf pairing 
\[
\left< , \right>_{\eta}:H^{-m}(B_{Ass}(X)) \otimes \pi_m(X)\rightarrow \K
\]
 as follows. Let $f:S^m \rightarrow X$ be a map and $\omega\in B_{Ass}(X)$ be a cocycle, then we define the Hopf pairing by 
\[
\left<\omega,f\right>_{\eta}=\int_{S^m} \tau(B_{Ass}f^*(\omega)),
\]
where $\tau(B_{Ass}f^*(\omega))$ is a weight one cocycle cohomologous to $B_{Ass}f^* (\omega)$ and the integral sign means the evaluation of the cocycle $\tau(B_{Ass}f^*(\omega))$ on the fundamental class of $S^m$.

\end{definition}

\begin{remark}
With the fundamental class of $S^m$ we mean a cycle representing the generator of the homology of $S^m$. If we assume that $S^m$ is triangulated, a choice for the fundamental class would be the sum of all $m$-simplices. Technically this definition depends on a  choice of orientation of the fundamental class of $S^m$, from now on we always assume that we have fixed an orientation on $S^m$.
\end{remark}

\section{$\E_n$-Hopf invariants}\label{secEnhopfinv}

In this section we give a generalization of the Hopf invariants from the previous section by replacing the associative bar construction by the $\E_n$-bar construction. More precisely, we define invariants of maps from $S^m$ to a space $X$, by defining a pairing between the homotopy groups of $X$ and the homology of the $\E_n$-bar construction on the cochains of $X$. Most of the arguments of this generalization  are similar to the arguments of \cite{SW2}, except for a few differences. The first difference is that we need some connectivity assumptions on $S^m$ for the arguments to work, in particular we need to assume that $m>n$. The second main difference is that for the weight reduction argument to work we need to introduce an additional grading on the bar complex which we will call the $\E_n$-degree. 

First we need to do a small computation on the bar construction.

\begin{lemma}\label{lemhomologyoftheenbarconstructionindegreen}
Let $S^m$ be the $m$-dimensional sphere with $m>n$, then $H^{-m}(B_{\E_n}(\N(S^m)))=\K$ and generated (as a $\K$-module) by $\H^{-m}(S^m)$.
\end{lemma}

\begin{proof}
To prove the lemma let $\mathfrak{S}^m$ be the simplicial model for the sphere with only one non-degenerate simplex concentrated in dimension  $m$. Observe that for every simplicial model of the sphere, there exists a zig-zag of quasi-isomorphisms from the cochains on that model to $\N(\mathfrak{S}^m)$, the cochains on the one dimensional model of $S^m$. Because of Proposition \ref{propbarconstructionpreservesqi}, the $\E_n$-bar construction preserves quasi-isomorphisms, we therefore have a zig-zag of quasi-isomorphisms between $B_{\E_n}\H^*(\N(S^m))$ and $B_{\E_n}\H^*(\mathfrak{S}^m)$. Because of Theorem \ref{thrmspheresaresomethimesformal}, the algebra $\N(\mathfrak{S}^m)$ is a one dimensio\-nal trivial $\E_n$-algebra and the $\E_n$-bar construction is therefore the cofree $\D_n$-coalgebra on one cogenerator of degree $-m$. Because the $\E_n$-structure on $\N(\mathfrak{S}^m)$ is trivial, the differential of the bar construction is just the internal differential coming from the $\D_n$-cooperad. Because of degree reasons the coalgebra $B_{\E_n}(\N(\mathfrak{S}^m))$ is one dimensional in degree $-m$ and generated by the generating cocycle of $\mathfrak{S}^m$, this is a weight $1$ cocycle. Since $\mathfrak{S}^m$ is trivial as an $\E_n$-algebra, the differential preserves weight and this cocycle is therefore also a cohomology class, which proves the lemma.
\end{proof}

It turns out that similar to Lemma \ref{lemcohomologyofthebarconstruction}, the homology of the $\E_n$-bar construction is also generated by the cohomology of $S^m$. The next step in the construction of the $\E_n$-Hopf invariants is to give an explicit construction that associates to every degree $-m$ cocycle in $B_{\E_n}\N(S^m)$ a cohomologous cocycle of weight one. The problem with the proof of Lemma \ref{lemhomologyoftheenbarconstructionindegreen} is that it does not give us an explicit method which produces a weight one cocycle from a cocycle $\omega \in B_{\E_n}(S^m)_{-m}$.  We will now give an explicit algorithm that, given a cocycle $\omega$ produces a weight one cocycle $\omega'$ cohomologous to $\omega$. By evaluating this cocycle on the fundamental class of $S^m$ we can create homotopy invariants of maps. To  define this algorithm we first need an additional grading on the $\E_n$-bar construction.

\begin{definition}
Let $A$ be an $\E_n$-algebra and $B_{\E_n}(A)$ the $\E_n$-bar construction on $A$. We define the $\E_n$-degree of a homogeneous element $\omega \in \D_n(r)_d \otimes_{\S_r}A^{\otimes r}$ as $d$.  
\end{definition}

Note that since $\D_n$ is finite dimensional in each arity (see Lemma \ref{lemDnisbounded}), there are only finitely many $\E_n$-degrees in which there are non-zero elements.  Further note that the differential of the bar construction does not increase the $\E_n$-degree.

We will now give the algorithm that reduces the weight of a cocycle to $1$. The main difference with the Sinha-Walter approach is that instead of immediately lowering the weight, we first lower the $\E_n$-degree as much as possible and then we can lower the weight.  


First, recall that the bar differential consists of three parts. The first part $d_1$, comes from the internal differential of $\N(S^m)$ and preserves the weight and $\E_n$-degree. The second part $d_2$, comes from the multiplication of $\N(S^m)$ and lowers the weight and does not increase the $\E_n$-degree. The third part $d_3$, comes from the differential of the $\D_n$-cooperad and preserves the weight, but lowers the $\E_n$-degree by $1$.

Suppose that we start with a degree $-m$ cocycle $\omega \in B_{\E_n}\N(X)$, which is given by $\sum \omega_{w,e}$ where $\omega_{w,e}$ is the component of $\omega$ which is of $\E_n$-degree $e$ and weight $w$. When we pull this back we get a cocycle $B_{\E_n}f^*\omega$ in $B_{\E_n}(\N(S^m))$. The highest weight, highest $\E_n$-degree component of $B_{\E_n}f^*\omega$ is then of the form $\sum_{i} \gamma_i \otimes \nu^i$, where $\nu^i \in \N(S^m)^{\otimes w}$ and $\gamma_i \in  \D_n(w)_e$ runs over a basis for $\D_n(w)$ as an $\S_w$-module. So in general $\nu^i$ is of the form $\nu^i= \sum_{j_1,..,j_w}  \nu^i_{j_1} \otimes ...\otimes \nu^i_{j_w}$, with  $\nu^i_{j_k} \in \N(S^m)$.  Because $d_2$ and $d_3$ lower the weight or $\E_n$-degree and $\omega$ is a cocycle, the element $\nu^i$ is a cocycle in $\N(S^m)^{\otimes w}$. Because of degree reasons it turns out that $\nu^i$ is exact in $\N(S^m)^{\otimes w}$. This can be seen as follows, first note that by the K\"unneth Theorem the homology of $\N(S)^{\otimes w}$ is concentrated in degree $-mw$. Next we need to show that the degree of $\nu^i$ is not equal to $-mw$, because that implies that it is exact.  Because $\omega$ is of degree $-m$ and $\gamma_i$ is of degree $e$, the sum of the degrees of the $\nu_j^i$ is equal to $-m-e$. Because $\D_n(w)$ is bounded, $e$ is at most $-n(w-1)$, which implies that $\vert \sum_{j_1,...,j_w} \nu_{j_1}^i \otimes... \otimes \nu_{j_w}^i \vert$ is of degree greater than $-m-n(w-1)$. Since we assumed that $m>n$, this is bigger than $-mw$, the element $\sum_{j_1,...,j_w}\nu^i_{j_1} \otimes ...\otimes \nu^i_{j_w}$ is therefore exact in $\N(S^m)^{\otimes w}$. 




Because the cocycle $\nu^i$ is exact in $\N(S^m)^{\otimes w}$, we can find an element $d_1^{-1}(\nu^i)\in \N(S^m)^{\otimes w}$. If we do this for all $i$, we construct an element $\sum_{i}\gamma_i\otimes d^{-1}_1 \nu^i$.  The element $\sum_i \gamma_i \otimes d^{-1}_1 \nu^i$ is then a boundary between $\omega$ and  $\omega-(d_1+d_2+d_3)(\sum_{i}\gamma_i\otimes d^{-1}_1 \nu^i)$, the weight $w$ part of this element has now a lower $\E_n$-degree. Because $\D_n(w)$ is bounded below we can iterate these steps until we reach the lower bound of $\D_n(w)$, at this point the differential $d_3$ will be zero and the construction described above will lower the weight. We can repeat these steps until we reach a weight one cocycle. Since $\D_n(w)$ is bounded we know that this process will eventually stop. We call the weight one cocycle we obtain this way the weight reduction of $B_{\E_n}f^*\omega$.


\begin{definition}
Let $f:S^m \rightarrow X$ be a map and $\omega\in B_{\E_n}\N(X)$ be a cocycle, we will denote by  $\tau(B_{\E_n}f^*\omega)$ a weight $1$ cocycle which is cohomologous to $B_{\E_n}f^*\omega$.
\end{definition}

\begin{proposition}\label{prophopfiswelldefined}
The cohomology class of $\tau(B_{\E_n}f^*\omega)$ is well defined.
\end{proposition}

\begin{proof}
To show that $\tau(B_{\E_n}f^*\omega)$ is a well defined cohomology class we need to show that different choices for $d_1^{-1}\nu^i$ give cohomologous cocycles.  Suppose that we have two different choices $\alpha$ and $\alpha'$ for $\sum_i \gamma_i\otimes d^{-1}_1(\nu^i)$, then it is easy to see that $\alpha$ and $\alpha'$ will differ by by a cocycle $\Lambda\in \N(S^m)^{\otimes w}$, i.e. $\alpha=\alpha'+\Lambda$. 

We want to show that the two weight reductions $\omega-(d_1+d_2+d_3)\alpha$ and $\omega-(d_1+d_2+d_3)\alpha'$ are cohomologous. Since $\alpha$ and $\alpha'$ differ by the cocycle $\Lambda$, the weight reductions $\omega-(d_1+d_2+d_3)\alpha$ and $\omega-(d_1+d_2+d_3)\alpha$ will differ by $(d_1+d_2+d_3)\Lambda$. Since $d_1 \Lambda=0$ we get that $(d_1+d_2+d_3)\Lambda$ is equal to $(d_2+d_3)\Lambda$. We therefore need to show that $(d_2+d_3)\Lambda$ is a coboundary in the $\E_n$-bar complex.

To do this first notice that because of similar arguments as before, the element $\Lambda$ is a cocycle of degree higher or equal to $-m-n(w-1)+1$ in $\N(S^m)^{\otimes w}$ which has only non-trivial homology in degree $-mw$, we can therefore find an element $\lambda\in \N(S^m)^{\otimes w}$, such that $d_1 \lambda =\Lambda$, i.e. $\Lambda$ is a boundary. The next step is to show that $\pm(d_2+d_3)\lambda$ is a boundary bounding $(d_2+d_3)\Lambda$. To do this we will compute $d_{B_{\E_n} A}(d_2+d_3)\lambda$, which is given by $(d_1+d_2+d_3)(d_2+d_3)\lambda$. Because of Lemma \ref{lemmadifferentialsanticommute} we see that this is equal to $(d_2+d_3)d_1\lambda$ and since $d_1\lambda=\Lambda$ this is equal to $(d_2+d_3)\Lambda$. The element $(d_2+d_3)\lambda$ is therefore a boundary bounding $(d_2+d_3)\Lambda$. The two different weight reductions  differ therefore by a boundary and give the same class in cohomology. The weight reduction is therefore well defined and the cohomology  class of the element $\tau(B_{\E_n}f^*\omega)$ does not depend on choices.



\end{proof}

\begin{remark}
 A way to construct an explicit representative of $d^{-1}_1(\nu^i)$ is as follows. Since all the degrees of the $\nu^i_{j_k}$ in $\sum_j \nu^i_{j_1} \otimes ... \otimes \nu^i_{j_w} $ are different from $-m$ we can can find an explicit choice of $d_A^{-1}\nu_{j_1}^i$. It turns out that $\sum_{j} d^{-1}_1\nu_{j_1}^i\otimes \nu_{j_2}^i \otimes...\otimes \nu_{j_w}^i$ is an explicit choice of $d_1^{-1}(\sum_j \nu^i_{j_1} \otimes ... \otimes \nu^i_{j_w} )$.
\end{remark}

\begin{remark}
Note that the proof of Proposition \ref{prophopfiswelldefined} and the weight reduction algorithm breaks down when we take $n\geq m$. In this case  the element $\Lambda$ could live in degree $-mw$ and could therefore represent a non-trivial homology class.
\end{remark}

\begin{definition}
Let $X$ be a space, the $\E_n$-Hopf pairing is defined as the pairing 
\[
\left<,\right>_{\E_n}:H_m(B_{\E_n}(C^*(X))\otimes \pi_m(X) \rightarrow \F
\]
 given by 
  \[
  \left<\omega,f \right>_{\E_n}:=\int_{S^m}\tau(B_{\E_n}f^*\omega),
  \]
  where $\omega\in H_*(B_{\E_n}(\N(X))$, $f:S^m \rightarrow X \in \pi_*(X)$, $\tau(B_{\E_n}f^*\omega)$ is a choice of weight one cocycle cohomologous to $\omega$ and the integral sign means the evaluation of $\tau(B_{\E_n}f^* \omega)$ on the fundamental class of $S^m$.
\end{definition}

\begin{theorem}\label{thrmEnhopfinv}
The $\E_n$-Hopf pairing is a well defined pairing and therefore an invariant of the homotopy class of a map $f:S^m \rightarrow X$.
\end{theorem}

\begin{proof}
To show that the Hopf pairing is well defined we need to show that if $f:S^m \rightarrow X$ and $g:S^m \rightarrow X$ are homotopic maps, then $\left<f,\omega \right>_{\E_n}=\left< g, \omega \right>_{\E_n}$ for all cocycles $\omega\in B_{\E_n}X$. Because of Proposition \ref{prophopfiswelldefined} the weight reduction only depends on the cohomology class of $B_{\E_n}f^*\omega$, we therefore need to show that $B_{\E_n}f^*\omega$  and $B_{E_n}g^* \omega$ are cohomologous for all $\omega$. To do this we will show that the maps $B_{\E_n}f^*$ and $B_{\E_n}g^*$ are chain homotopic, because this implies that they map cohomologous cocycles to cohomologous cocycles. 

To show that the maps $B_{\E_n}f^*,B_{\E_n}g^*:B_{\E_n}\N(X)\rightarrow B_{\E_n}\N(S^m)$  are chain homotopic, first note that $f^*,g^*:\N(X) \rightarrow \N(S^m)$ are homotopic as maps of $\E_{\infty}$-algebras. Therefore there exists a good path object $\N(S^m)^I$ and a map $H:\N(X) \rightarrow \N(S^m)^I$ which is a homotopy between $f^*$ and $g^*$. Since the bar construction preserves quasi-isomorphisms and maps surjective maps to surjective maps, the $\D_n$-coalgebra $B_{\E_n}(\N(S^m)^I)$ is a good path object for $B_{\E_n}\N(S^m)$ in the model category of chain complexes. The map $B_{\E_n}H:B_{\E_n}\N(X) \rightarrow B_{\E_n}(\N(S^m)^I)$ is then a homotopy between $B_{\E_n}f^*$ and $B_{\E_n}g^*$ in the model category of chain complexes. The maps    $B_{\E_n}f^*$ and $B_{\E_n}g^*$ are therefore homotopic in the category of chain complexes and induce the same map in cohomology. The $\E_n$-Hopf pairing is therefore well defined.
\end{proof}

\begin{remark}
This proof could be greatly simplified if we had a model structure on the category of coalgebras with divided symmetries and could use the results from \cite{DCH1} and \cite{Val1}. In this case we could use the fact that the $\E_n$-bar construction is a left Quillen functor. Unfortunately, their are certain technical differences between the assumptions of \cite{DCH1} and this paper. It seems very likely however that all the results form \cite{DCH1} also hold in our setting, but proving them would be beyond the scope of this paper.  
\end{remark}

\part{Suspensions}


The main goal of this part of the paper is to study the suspensions of $\E_{\infty}$-algebras and simplicial sets and their relation to the $\E_n$-Hopf invariants. 

In \cite{BF1}, Berger and Fresse defined a morphism $\sigma:\E_\infty \rightarrow  \Sus^{-1} \E_\infty$, which induces a functor from $\E_\infty$-algebras to $\E_\infty$-algebras by sending an $\E_\infty$-algebra $A$ to its desuspension $s^{-1}A$, with a new $\E_\infty$-structure. We denote  this functor by $\sigma^{-1}$, and the suspension of $A$ with the new $\E_\infty$-algebra structure by $\sigma^{-1} A$. In \cite{BF1}, it is also shown that for spheres this functor has the property that it sends  the reduced normalized cochains on the one dimensional model for $S^m$ to the reduced normalized cochains of the one dimensional model for $S^{m+1}$.

The first main result of this part is that we extend the results of Berger and Fresse and show that the  $\E_\infty$-algebra suspension functor commutes with the suspension of simplicial sets, i.e. we have a natural isomorphism between $\sigma^{-1} \N(X)$ and $\N(\Sigma X)$, where $X$  is a simplicial set, $\Sigma X$ the simplicial suspension of $X$ and $\sigma^{-1} \N(X)$ the $\E_\infty$-suspension of $\N(X)$. Dually we also have an isomorphism of $\E_\infty$-coalgebras for the chains, i.e. we have a natural isomorphism  $\sigma \NN(X) \cong \NN ( \Sigma X)$ of $\E_\infty$-coalgebras.

The second result of this part is that we determine the relationship between the $\E_n$-Hopf invariants of a simplicial set $X$ and the $\E_{n+1}$-Hopf invariants of its suspension $\Sigma X$.

\begin{remark}
Note that because our grading conventions the suspension of a simplicial set will induce a desuspension on the level of reduced normalized cochains. So technically the suspension functor should be called a desuspension functor when applied to cochains. For chains it will still be a suspension functor.
\end{remark}


\section{Suspensions of $\E_{\infty}$-algebras}\label{secsuspensions}

In this section we recall the definitions of the suspension morphism for $\E_\infty$-algebras and the suspension of a simplicial set. Then we use these definitions to show that there is a natural isomorphism between $\sigma^{-1} \N(X)$ and $\N(\Sigma X)$.

\subsection{The suspension morphism}

In this subsection we recall some facts about suspensions of $\E_{\infty}$-algebras coming from \cite{Fres3} and \cite{BF1}. These facts are important in the next sections where we will use them to give a relation between the $\E_n$-Hopf invariants and the $\E_{n+1}$-Hopf invariants.


To define the suspension morphism we first need to introduce a special cochain called $sgn$, which is a function $sgn:\E_{\infty}(r)_{r-1} \rightarrow \K$ of degree $r-1$ and is defined as follows. Let $(w_0,...,w_{r-1})$ be an element of $\E_{\infty}(r)_{r-1}$, with each $w_i \in \S_r$. The cochain $sgn$ is then defined to be zero if the sequence $(w_0(1),...,w_{r-1}(1))$ does not form a permutation of the set $\{1,...,r\}$, otherwise it is defined as the sign of this permutation, i.e.
\[
sgn(w_0,...,w_{r-1})=sgn(w_0(1),...,w_{r-1}(1)).
\] 
 Using this cochain we can now define the suspension morphism for the Barratt-Eccles operad.

\begin{definition}
 The suspension morphism for the Barratt-Eccles operad $\sigma:\E_{\infty} \rightarrow \Sus^{-1} \E_\infty$ is defined as follows. Let $(w_0,...,w_d)\in \E_\infty(r)_d$ then we define the suspension morphism $\sigma$ by
 \[
 \sigma(w_0,...,w_d)=sgn(w_0,...,w_{r-1}) \cdot (w_{r-1},...,w_d).
 \]
\end{definition}

Similar to the Barratt-Eccles operad, the surjection operad $\X_{\infty}$ also has a suspension morphism which is defined as follows.

\begin{definition}
The suspension morphism $\sigma:\X_\infty \rightarrow \Sus^{-1}\X_\infty$ for the surjection operad $\X_\infty$ is defined as follows. Recall that we can represent each surjection $u:\{1,..,r+d\} \rightarrow \{1,..,r\}$ by the sequence $(u(1),...,u(r+d))$. The suspension morphism is then defined by
\[
\sigma(u)=sgn(u(1),...,u(r))(u(r),...,u(r+d)),
\]
where $sgn(u(1),...,u(r))$ is the sign of the permutation $(u(1),...,u(r))$ if all the $u(i)$ are distinct and zero otherwise.
\end{definition}

As is explained in Remark 3.2.10 of \cite{BF1}, the suspension morphisms of $\E_\infty$ and $\X_\infty$ are related as follows.

\begin{lemma}\label{lemcommutaivediagramwithsuspensionmorphisms}
The suspension morphisms fit in the following commutative diagram
\[
\xymatrix{
\E_\infty \ar[d]^{\sigma} \ar[r]^ {TR}  & \X_\infty \ar[d]^{\sigma} \\
\Sus^{-1} \E_\infty \ar[r]^{\Sus^ {-1} TR} & \Sus^ {-1} \X_\infty.
}
\]
\end{lemma}

 Because of this diagram we only need to prove statements for the surjection operad and the corresponding statement for the Barratt-Eccles operad will automatically follow.


\subsection{The $\E_\infty$-suspension functor}\label{defXsuspsnesionmorphism}

The suspension morphisms induce  the following four functors 
\[
\sigma:\E_\infty\mbox{-coalgebras} \rightarrow \E_\infty\mbox{-coalgebras},
\]
\[
\sigma:\X_\infty\mbox{-coalgebras} \rightarrow \X_\infty\mbox{-algebras},
\]
\[
 \sigma^{-1}:\E_\infty\mbox{-algebras} \rightarrow \E_\infty\mbox{-algebras}
\] 
and 
\[
\sigma^{-1}:\X_\infty\mbox{-algebras} \rightarrow \X_\infty\mbox{-algebras}
\]
 which are defined as follows. First note that for an $\E_{\infty}$-algebra $A$, the desuspension (as a chain complex)  $s^{-1} A$, is an algebra over the desuspended $\E_\infty$-operad $\Sus^{-1}\E_\infty$. We can then use the suspension morphism to induce an $\E_\infty$-algebra structure on $s^{-1} A$. So we define a functor $\sigma:\E_\infty\mbox{-algebras} \rightarrow \E_\infty\mbox{-algebras}$, which sends an $\E_\infty$-algebra $A$ to the $\E_\infty$-algebra $s^{-1} A$ with the  $\E_{\infty}$-algebra structure coming from the suspension morphism. We will denote $s^{-1} A$ with this new $\E_\infty$-algebra structure by $\sigma^{-1} A$.

 
Similarly we can also define a suspension functor for $\X_\infty$-algebras, by abuse of notation we will denote both these functors by $\sigma^{-1}$, from the context it should be clear which one we mean. 
 
For $\E_\infty$-coalgebras we also get a suspension functor. This functor is defined by sending a coalgebra $C$ to $sC$ with the $\E_\infty$-structure coming from the suspension morphism. Note that because we are working with coalgebras over operads, we get a pull back instead of the usual push forward functor. Similarly we get the same functor for $\X_\infty$-coalgebras. 
 
When we explicitly evaluate the suspension functor on an  $\X_\infty$-coalgebra $C$, then we get the following formulas for the coproducts of $\sigma C$. Let $sx\in sC$ and $u \in \X_{\infty}(r)$ and denote the $\sigma(u)$-coproduct of $x$ by $\sigma(u)(x)=\sum x_1 \otimes ... \otimes x_r$. The coproduct $u(sx)$ is then given by 
\[
u(sx)=\sum \pm s x_1 \otimes ... \otimes s x_r,
\]
where the sum ranges over the $\sigma(u)$-coproduct of $x$. The sign is given by 
\[(-1)^{ \vert \sigma(u) \vert + \vert x_1 \vert (r-1) + \vert x_2 \vert (r-2)+....+ \vert x_{r-1} \vert \cdot 1}.
\]
Dually, we can do the same for an $\X_{\infty}$-algebra $A$, in this case we get the following formulas for the products of $\sigma^{-1} A$. Let $s^{-1}x_1,...,s^{-1}x_r \in s^{-1}A$ and $u \in \X_{\infty}(r)$, then we define the product $u(s^{-1}x_1,...,s^{-1}x_r)$ by the following formula

\[
u(s^{-1}x_1,...,s^{-1}x_r):=\pm s^{-1} \sigma(u)(x_1,...,x_r),
\]
the sign $\pm$ is given by $(-1)^{r \cdot \vert \sigma(u) \vert+ \vert x_1 \vert (r-1) + \vert x_2 \vert (r-1)....+ \vert x_{r-1} \vert \cdot 1 }$. Similar to the $\X_{\infty}$-case we get the same formulas for the $\E_{\infty}$-case.  
 
In the rest of this paper we will also need iterated suspension morphisms. The $n$-fold iteration $\sigma^{-1} (\sigma^{-1} (... \sigma ^{-1} A)))$ of $\sigma^{-1}$  on $A$ will be denoted by $\sigma^{-n} A$.  

\begin{proposition}\label{propsuspensionpreserversqi}
The suspension functor $\sigma$ preserves quasi-isomorphisms and fibrations of $\E_\infty$-algebras.
\end{proposition}

\begin{proof}
Since both fibrations and quasi-isomorphism are defined as properties of the underlying chain complexes and the suspension of chain complexes preserves quasi-isomorphism and fibrations of chain complexes, the suspension functor $\sigma^{-1} $ also preserves  fibrations and quasi-isomorphisms. 
\end{proof}

The following lemma can be found as the Fact in Section 7 of \cite{Fres3}.

\begin{lemma}\label{lemspheressuspensions}
 Denote by $\N(\mathfrak{S}^m)$ the reduced normalized chains on $\mathfrak{S}^m$, the simplicial model for $S^m$ with one non-degenerate simplex of dimension $m$. Then there is an isomorphism of $\E_\infty$-algebras between $\sigma^{-1} \N(\mathfrak{S}^m)$ and  $\N(\mathfrak{S}^{m+1})$.
\end{lemma}

\subsection{The reduced suspension of a simplicial set}\label{subsecreducedsuspensions}

Before we can compare the $E_\infty$-algebra suspension functors $\sigma$ and $\sigma^{-1}$ with the reduced suspension of a simplicial set, we first recall the definition of the reduced suspension of a simplicial set. For more details see \cite{CurSS}, Definition 1.9.

\begin{definition}
Let $X$ be a based simplicial set with base point $*$. Then we define the simplicial cone $CX$ on $X$ as the following simplicial set. The set of $n$-simplices is given by $CX_n=\{(x,j) \mid  x \in X_{n-j}, 0 \leq j \leq n  \} $, with all $(*,j)$ identified to $*$. The face and degeneracy maps are given by 
\[
   d_i(x,j)= 
\begin{cases}
   (x,j-1)            & \text{for } 0\leq i < j, \\
   (d_{i-j}(x),j)    & \text{for } j \leq i \leq n,
\end{cases}
\]

\[
   s_i(x,j)= 
\begin{cases}
   (x,j+1)            & \text{for } 0\leq i < j, \\
   (s_{i-j}(x),j)    & \text{for } j \leq i \leq n.
\end{cases}
\]

The reduced suspension $\Sigma X$ is defined as the quotient of $CX$ by $X$, where $X$ is identified with the simplicial subset of $CX$ given by all simplices of the form $(x,0)$.
\end{definition}

It is straightforward to check that that the set of non-degenerate simplices of $\Sigma X$ is given by all simplices of the form $(x,1)$, such that $x$ is a non-degenerate simplex of $X$ (and the base point). Using this bijection we get the following isomorphism. 

\begin{lemma}\label{lemphiisanisomorphism}
There is a natural isomorphism of chain complexes
$$\phi:s \tilde{N}_*(X)\rightarrow \tilde{N}_*(\Sigma X),$$
from $s \tilde{N}_*(X)$, the  suspension (as chain complexes) of the reduced normalized chains on $X$, to $\tilde{N}_*(\Sigma X)$ the reduced normalized chains on the reduced suspension of $X$, given by 
\[
\phi( s x)=(x,1)
\]
Dually, there is also a natural isomorphism
\[
\varphi:s^{-1} \N ( X) \rightarrow  \N(\Sigma X),
\]
 between the cochains on the simplicial suspension of $X$ and the desuspension of the cochains of $X$. This is given by sending a cochain $s^{-1}\alpha \in s^{-1} \N(X)$, for $\alpha \in \N(X)$, to the cochain $\varphi (\alpha)$ defined by $\varphi(s^{-1} \alpha)(x,1): =\alpha (x)$, for $x$ a non-degenerate simplex in $\tilde{N}_*(X)$.
\end{lemma}                 

The proof of the lemma is straightforward and left to the reader.             

For the next section we will also need a description of the non-degenerate $(n+1)$-simplex of $\Sigma \Delta^n$. Recall from \cite{CurSS} Example 1.4, that the standard $n$-simplex can be described as 
\[
(\Delta^n)_q=\{(v_0,...,v_q)\mid v_i \in \{0,...,n\} \mbox{ and } 0\leq v_0 \leq ... \leq v_q \leq n\}.
\]
A subsimplex $(v_0,...,v_q)$ of the standard simplex is non-degenerate if $v_0<v_1<...<v_q$. It can easily be seen that the cone $C \Delta^n$ has exactly one non-degenerate $(n+1)$-simplex, which is given by $((0,1,...,n),1)$, where $(0,1,...,n)$ is the non-degenerate $n$-simplex of $\Delta^n$. Similarly the suspension $\Sigma \Delta^n$ also has only one non-degenerate $(n+1)$-simplex.  The map $\tilde{\psi}:\Delta^{n+1} \rightarrow C\Delta^n$  from $\Delta^{n+1}$ to $C\Delta^n$ is then defined as follows. If $x\in \Delta^{n+1}$ is a simplex of the form $(v_0,...,v_n)$ such that $v_0$ till $v_{q-1}$  are equal to $0$, then we send this simplex to $(v_q-1,...,v_{n+1}-1),q)\in C\Delta^n$. It is straightforward to check that this is indeed a map of simplicial sets and we leave this to the reader. 

\begin{definition}\label{defthemappsi}
The map $\psi:\Delta^{n+1} \rightarrow \Sigma \Delta^n$ is defined as the composition of $\tilde{\psi}$ with the quotient map form $CX$ to $\Sigma X$.
\end{definition}

In the following lemma we will describe the simplices of $\Delta^ {n+1}$ which are mapped to non-degenerate simplices under the map $\psi$.


\begin{lemma}\label{lemnondegeneratesimplicesunderpsi}
Let $y \subset \Delta^{n+1}$ be a $k$-dimensional non-degenerate subsimplex of $\Delta^ {n+1}$. Then the subsimplex $\psi(y) \subset \Sigma \Delta^n$  is non-degenerate if and only if $y$ is of the form $(0,a_1,...,a_k)$ with $1 \leq a_1 < ... < a_k\leq n$. 
\end{lemma}

\begin{proof}
It is easy to see that all the non-degenerate simplices of $\Sigma \Delta^n$ are of the form $(x,1)$, where $x$ is a non-degenerate simplex of $\Delta^n$.  The only simplices which map to a simplex of the form $(x,1)$ are the simplices which start with exactly one $0$. The condition for $x$ to be non-degenerate is given by $1\leq  a_1 < ... < a_k\leq n$, which is exactly the condition given by the lemma.
\end{proof}
                  
\subsection{The relation between the suspension of simplicial sets and the suspension of $\mathcal{E}_{\infty}$-algebras}

In the previous section we have seen that there is an isomorphism of chain complexes between $s^{-1} \tilde{N}^*(X)$ and $\tilde{N}^*(\Sigma X)$. In this section we will show that this isomorphism is not just an isomorphism of chain complexes, but when we equip $s^{-1}\N(X)$ with the $\E_\infty$-structure coming from the suspension morphism, this is also an isomorphism of $\E_{\infty}$-algebras.
                   
\begin{theorem}\label{thrmsuspensionfunctor}
Let $X$ be a simplicial set. The map $\phi$ from Lemma \ref{lemphiisanisomorphism} is a natural isomorphism of $\X_\infty$-coalgebras 
\[
\phi:\sigma\tilde{N}_*(X) \rightarrow \tilde{N}_*(\Sigma X),
\]
where $\sigma \tilde{N}_*(X)$ is the $\X_\infty$-suspension of the reduced normalized chains on $X$ and $\tilde{N}_*(\Sigma X)$ are the reduced normalized chains on the simplicial suspension of $X$.

Dually, the map \[
\varphi:\sigma^{-1}\tilde{N}^*(X) \rightarrow \tilde{N}^*(\Sigma X)
\]
is a natural isomorphism of $\X_\infty$-algebras.
\end{theorem}

\begin{proof}
Because of Lemma \ref{lemphiisanisomorphism}, we know that the maps $\phi:\Sigma \tilde{N}_*(X) \rightarrow \tilde{N}_*(\Sigma X)$ and $\varphi:\sigma^{-1}\tilde{N}^*(X) \rightarrow \tilde{N}^*(\Sigma X)$  are natural isomorphisms of chain complexes. We need to show that they are also isomorphisms of $\X_\infty$-(co)algebras. To do this we need to show that the isomorphisms $\phi$ and $\varphi$ commute with the $\X_{\infty}$-(co)algebra structure maps.

For simplicity we will only show that $\phi$ is an $\X_\infty$-coalgebra map, since the arguments for $\varphi$ being an $\X_\infty$-algebra map are dual.  Let $u \in \X_{\infty}(r)_d$ be a surjection, then we need to show that the following diagram commutes

\begin{equation}\label{diagram23}
\xymatrix{
 s \NN(X) \ar[r]^{\sigma (u)} \ar[d]^{\phi} & (s \NN(X))^{\otimes r}  \ar[d]^{\phi^{\otimes r}}\\
 \NN(\Sigma X) \ar[r]^{u} & (\NN(\Sigma X))^{\otimes r}.
}
\end{equation}

To show that the diagram commutes, we first show that it commutes when $X$ is $\Delta^n$. We can then extend the result to general simplicial sets as follows. Let $x\in X$ be a non-degenerate $n$-simplex, then we can describe $x$ as a map $x:\Delta^n \rightarrow X$. The suspension of the simplex $x$ is then an $(n+1)$-simplex $\Sigma x \in \Sigma X$, which is described by the map $\Sigma x: \Delta^{n+1} \xrightarrow{\psi} \Sigma \Delta^n \xrightarrow{\Sigma x} \Sigma X$, by abuse of notation we have denoted both the suspension of $x:\Delta^n \rightarrow X$ and the map from $\Delta^{n+1}$ to $\Sigma X$ by $\Sigma x$. So every $(n+1)$-simplex in the suspension of $X$ factors through $\Sigma \Delta^n$, the suspension of the standard $n$-simplex. We therefore need to understand  the coproduct of $\Sigma \Delta^{n}$ and compare it to the coproduct of $\Delta^{n+1}$.

According to Section \ref{seccochainEinftyalg}, the coproduct of a simplex $x$ is defined by applying the map $x:\Delta^n\rightarrow X$ to the coproduct of $\Delta^n$. So if we can show that Diagram \ref{diagram23} commutes for $\Delta^n$, it will commute for all simplicial sets $X$.




Recall from Section \ref{seccochainEinftyalg}, that the coproduct $u:\tilde{N}_{*} (X) \rightarrow \tilde{N}(X)_* ^{\otimes r}$ is given by the sum over all interval  cuts of the interval $[0,n]$ with $d+r$ pieces. More precisely it is given by
\[
u(x)=\sum \pm  \Delta (C_{(1)})\otimes ... \otimes \Delta (C_{(r)}),
\]
where the $\Delta (C_{(i)})$ are the simplicial subsets of $\Delta^ {n+1}$ defined in Section \ref{seccochainEinftyalg}.

Because the simplicial set $\Sigma \Delta^n$ is a suspension, a majority of the non-degenerate simplices of $\Delta^{n+1}$ will get mapped to a degenerate simplex under the map $\psi:\Delta^{n+1} \rightarrow \Sigma \Delta^n$. As is explained in Lemma \ref{lemnondegeneratesimplicesunderpsi}, the map $\psi$ maps every subsimplex of $\Delta^ n$ that is not of the form $(0,a_1,...a_k)$, with $1 \leq a_1 < .. < a_k\leq n+1$, to a degenerate simplex. The only terms that might give a non-zero contribution to the coproduct of $\Sigma \Delta^{n}$ are therefore only  the interval cuts of $[0,n+1]$ that start with the exactly one zero for each value of $\{1,...,r\}$. This requires the surjection $u$ to start with a permutation of the set $\{1,...,r\}$, because otherwise the corresponding simplices in the interval cuts are degenerate. The remainder of the interval cut is an interval cut with $d+1$ pieces of the interval $[1,n+1]$. The idea of the proof is that we identify the interval cut of $[1,n+1]$ with the interval cut associated to $\sigma(u)$ of the interval $[0,n]$.





More precisely, for a term $\psi(\Delta (C_{(1)}))\otimes ... \otimes \psi(\Delta (C_{(k)}))$ in the coproduct $u(\Sigma \Delta^n)$ to be non-zero we need that all the $\psi(C_{(i)})\subset \Sigma \Delta^n$ to be non-degenerate. Because of Lemma \ref{lemnondegeneratesimplicesunderpsi}, this can only happen when all the sequences $C_{(i)}$ start with exactly one $0$. The sequence corresponding to the surjection $u$ should therefore start with a permutation of $r$. If it does not start with a permutation of $r$ then for each $C_{(i)}$ to start with a $0$ would imply that there would be at least one $C_{(i)}$ with starts with several zeros. It is then straightforward to check that the part $[1,n+1]$ of the interval cut of $[0,n+1]$ is exactly the one corresponding to the interval cut of $[0,n]$ associated to the surjection $\sigma(u)$. So up to a sign we have shown that $\phi^{\otimes r}( \sigma (u)(\Delta^n))=u(\phi(\Delta^n)$.

To show that the signs match as well we will first compute the sign of the map 
\[
s\NN(X)\xrightarrow{\sigma(u)}(s \NN (X))^{\otimes r} \xrightarrow{\phi^{\otimes r}} (\NN(\Sigma X))^{\otimes r}.
\]
 Fix an interval cut of $[0,n]$ with $\vert \sigma (u)\vert+r=d+1$ pieces and denote the cut points by $0\leq n_1 \leq ... \leq n_d \leq n_{d+1}=n$. Because of Definition \ref{defXsuspsnesionmorphism} and Section \ref{seccochainEinftyalg}, the sign corresponding to this interval cut is given by  $sgn(u(1),...,u(r))\cdot(-1)^{\vert \sigma(u) \vert+(x_1)(r-1)+x_2(r-2)+...+x_{r-1}\cdot 1+Perm+Pos}$, where $x_i$ denotes the dimension of  $\Delta (C_{(i)})$, $Perm$ is the permutation sign and $Pos$ is the position sign. We need to compare this with the sign of the map $s\NN(\Delta^n)\xrightarrow{\phi} \NN (\Sigma \Delta^n) \xrightarrow{u} (\NN(\Sigma \Delta^n))^{\otimes r}$. 

To do this, note that the $[1,n+1]$ part of the interval cut of the interval $[0,n+1]$  coincides with the interval cut of $[0,n]$ corresponding to $\sigma(u)$. The position sign corresponding to $u$ is given by the sum of the end points of the inner intervals. Note that the first $r-1$ inner intervals all end at $0$ and therefore do not contribute to the position sign. The other inner intervals are all shifted by one, since there are in total $d-r+1$ remaining inner intervals, the position sign corresponding to $u$ is given by $Pos+d-r+1$, where $Pos$ is the position sign of the interval cut associated to $\sigma(u)$. 

To compute the permutation sign of the interval cut associated to $u$, we first rearrange the interval $[0,1]$ and the interval $[1,n+1]$. The sign corresponding to rearranging the interval $[0,1]$ is exactly the sign of the permutation $(u(1),...,(u(r))$, the sign corresponding to rearranging the interval $[1,n+1]$ is exactly the Permutation sign $(-1)^{Perm}$ from the interval cut corresponding to $\sigma(u)$. When we shuffle the intervals $[0,1]$ and $[1,n+1]$ we get another $(-1)^{x_1(r-1)+...x_{r-1} \cdot 1}$ sign.

When we combine all these signs and use the fact that $\vert \sigma(u) \vert =d-r+1$, we see that both these the signs match. The map $\phi$ is therefore an isomorphism of $\X_{\infty}$-coalgebras.
\end{proof}

\begin{remark}
Note that we did not use the assumption that we are working over a field to prove Theorem \ref{thrmsuspensionfunctor}. This Theorem therefore also holds over more general rings.
\end{remark}

\begin{corollary}
The map $\phi$ is an isomorphism of $\E_\infty$-coalgebras and the maps $\varphi$ is an isomorphism of $\E_{\infty}$-algebras.
\end{corollary}

\begin{proof}
This follows from Lemma \ref{lemcommutaivediagramwithsuspensionmorphisms} and Theorem \ref{thrmsuspensionfunctor}, since we have an isomorphism of $\X_\infty$-algebras between $\N(\Sigma(X)$ and $\sigma^{-1}(\N(X))$ and the diagram from Lemma \ref{lemcommutaivediagramwithsuspensionmorphisms} commutes we also have an isomorphism after we pull things back along the table reduction morphism.
\end{proof}
                    
\section{The relation between $\E_n$- and $\E_{n+1}$-Hopf invariants}                    

In this section we will use the suspension morphism of $\E_\infty$-algebras to give a relation between the $\E_n$-Hopf invariants of a simplicial set $X$ and the $\E_{n+1}$-Hopf invariants of the suspension of $X$. To do this we will first need to recall some more results from \cite{Fres3} about the Koszul duality of $\E_n$-operads. The main result of this section is Theorem \ref{thrmenen+1formula}.
                    
\subsection{Koszul duality for $\E_n$-operads: part 2}\label{subsecKoszulduality2}

For the applications we need a more detailed recollection of the results of \cite{Fres3}. In particular, we need to know more about the role of the suspension morphism in the Koszul duality results of \cite{Fres3}.
               
According to Theorem A of \cite{Fres3}, we have the following diagram of Koszul twisting morphisms, where $\D_n$ is defined as $\mathcal{S}^{-n}\E_n^{\vee}$ (see Section \ref{subsecKoszulduality1}).

\begin{equation}\label{diagram1}
\xymatrix{
\D_1\ar[d]^{\tau_1} \ar[r]^{\mathcal{S}^{-1}\sigma^{\vee}} & \D_2 \ar[d]^{\tau_2} \ar[r]^{\mathcal{S}^{-2}\sigma^{\vee}} & \D_3 \ar[d]^{\tau_3} \ar[r]^{\mathcal{S}^{-3}\sigma^{\vee}} &... \ar[r] & \D_{\infty} \ar[d]_{\tau_{\infty}} \\
 \E_1 \ar[r]^{\iota} & \E_2 \ar[r]^{\iota}&  \E_3 \ar[r]^{\iota}&...\ar[r]& \E_{\infty} 
}                           
\end{equation}
\begin{remark}
Strictly speaking the results proven in \cite{Fres3}, are stated in terms of a sequence of quasi-isomorphisms $\psi_n:\Omega \D_n\rightarrow \E_n$. Because of the Theorem 6.5.7 in \cite{LV}, this is equivalent to a sequence of Koszul twisting morphisms $\tau_n:\D_n \rightarrow \E_n$.
\end{remark}

For the applications we have in mind, we need the dual version of Diagram \ref{diagram1}. Recall that the suspension morphism is a morphism from $\E_{\infty}$ to $\mathcal{S}^{-1}\E_{\infty}$, which when restricted to $\E_n$ defines a map from $\E_{n}$ to $\mathcal{S}^{-1}\E_{n-1}$. One of the results of \cite{Fres3}, is that Fresse interprets the linear dual of the suspension morphism as the Koszul dual of the inclusion of the $\E_n$-operad into the $\E_{n+1}$-operad. For our applications we need the Koszul dual of the suspension morphism. It turns out that this is a morphism $\Upsilon:\mathcal{S}\D_{n+1} \rightarrow \D_{n} $ which is given by the (shifted) dual of the inclusion map $\iota:\E_{n}\rightarrow \E_{n+1}$.

\begin{definition}\label{defupsilon'}
The morphism $\Upsilon:\mathcal{S}\D_{n+1} \rightarrow \D_n$ is given by the shifted dual of the inclusion. More precisely, since $\D_n=\mathcal{S}^{-n}\E_n^{\vee}$, the inclusion map $\E_n\rightarrow \E_{n+1}$ defines a morphism from $\E_{n+1} ^{\vee} \rightarrow \E_n^ {\vee}$. After taking the $n$-fold desuspension of $\iota^{\vee}$,  we get the morphism $\Upsilon:\mathcal{S}^{-n}\E_{n+1}^{\vee}=\mathcal{S}\D_{n+1}\rightarrow \mathcal{S}^{-n}\E_n^{\vee}=\D_n$.
\end{definition}

The relation between $\Upsilon$ and the suspension morphism $\sigma$ is summarized in the following proposition.


\begin{proposition}\label{propthecommutingdiagramwiththemapformerlyknownasupsilon}
The following diagram commutes.
\[
\xymatrix{
\D_1\ar[d]^{\tau_1^{\vee}}  & \mathcal{S} \D_2 \ar[l]_{\Upsilon}\ar[d]^{\tau_2^{\vee}} & \mathcal{S}^{2} \D_3 \ar[d]^{\tau_3^{\vee}} \ar[l]_{\Upsilon} & \mathcal{S}^{3} \D_4 \ar[d]^{\tau_4^{\vee}} \ar[l]_{\Upsilon} & \mathcal{S}^{4} \D_5 \ar[d]^{\tau_5^{\vee}} \ar[l]_{\Upsilon} & ... \ar[l]_{\Upsilon} \\
 \E_1  & \mathcal{S} E_2 \ar[l]_{ \sigma}& \mathcal{S}^{2} \E_3 \ar[l]_{\sigma}& \mathcal{S}^{3} \E_4 \ar[l]_{\sigma}& \mathcal{S}^{4} \E_5 \ar[l]_{\sigma}&...\ar[l]_{\sigma}
   }                           
  \]             
\end{proposition}

\begin{remark}
Technically we are using the shifted version of the suspension morphism and of $\Upsilon$, for notational simplicity we have decided not to include this shift in the notation.
\end{remark}       

\begin{proof}
Recall from \cite{Fres3} that the suspension morphism and the inclusion $\iota:\E_n \rightarrow \E_{n+1}$ fit in the following diagram:
\[
\xymatrix{
\D_n \ar[r]^{S^{-n}\sigma^{\vee}} \ar[d]^{\tau_n} & \D_{n+1} \ar[d]^{\tau_{n+1}} \\
\E_n \ar[r]^{\iota} & \E_{n+1}.
}
\]
When we dualize this diagram and then suspend $(n+1)$-times, we get the desired diagram. Note that because of Lemma 7.4 in \cite{RNW1}, the dual of a twisting morphism is again a twisting morphism.
\end{proof}   

\begin{remark}
Since we do not have an explicit description of $\tau_n:\D_n \rightarrow \E_n$ or $\tau^{\vee}_n:\D_n \rightarrow \E_n$ we will denote them both by $\tau_n$. 
\end{remark}

\subsection{The map on the bar constructions induced by the suspension morphism}
The goal of this section is to use the map $\Upsilon$ to define a degree $1$ chain map from $B_{\E_{n+1}}(\sigma^{-1} A)$ to $B_{\E_n}A$.


\begin{definition}\label{defupsilon}
Let $A$ be an $\E_{\infty}$-algebra, then the map $\Upsilon:B_{\E_{n+1}}(\sigma^{-1} A)\rightarrow B_{\E_n}A$ is defined as follows. Recall that the underlying graded vector space of the $\E_n$-bar construction is given by $\D_n(A)=\bigoplus_r \D_n(r) \otimes_{\S_r} A^{\otimes r}$. Then we define the morphism 
\[
\Upsilon:B_{\E_{n+1}}(\sigma^{-1} A)\rightarrow B_{\E_n}A
\]
 as a morphism of graded vector spaces by sending 
\[
\Upsilon(\mu\otimes s^{-1} a_1 \otimes ... \otimes s^{-1} a_r)=\Upsilon(\mu)\otimes a_1 \otimes ... \otimes a_r,  
\]
where $\mu \in \D_n(r)$, $a_1,...,a_r \in A$ and $\Upsilon$ is the morphism from Definition \ref{defupsilon'}.
\end{definition}

\begin{proposition}\label{propupsilon}
The map $\Upsilon:B_{\E_{n+1}}(\sigma^{-1} A)\rightarrow B_{\E_n}A$ from Definition \ref{defupsilon} is a chain map, i.e. it commutes with the bar differentials.  
\end{proposition}

\begin{proof}
The proof of Proposition \ref{propupsilon} is a straightforward check that $\Upsilon$ commutes with the bar differentials. Recall that the bar differential $d_{B_{\E_{n}}}=d_1+d_2+d_3$ consists of three parts, $d_1$ comes from the differential of $A$, $d_2$ comes from the multiplication of $A$ and $d_3$ comes from the differential of the $\D_n$-cooperad. We need to show that $\Upsilon$ commutes with $d_1$, $d_2$ and $d_3$. Since the functor $\sigma^{-1}$ only shifts the underlying chain complex $A$, it is clear that  $\sigma^{-1}$ commutes with $d_1$. Similarly since $\Upsilon:\mathcal{S}\D_{n+1} \rightarrow \D_n$ is a chain map it commutes with the differentials of $\D_n$ and $\Sus\D_{n+1}$. So in particular the induced map will commute with the differentials, therefore $\Upsilon $ commutes with $d_3$. The hardest part is to show that $\Upsilon$ commutes with $d_2$.

To show that $\Upsilon$ commutes with $d_2$ first recall from Section \ref{secthebarconstruction}, that $d_2:B_{\E_{n+1}}\sigma^{-1} A \rightarrow B_{\E_{n+1}}\sigma^{-1} A$ is defined as 
\[
d_2:\D_{n+1} \circ \sigma^{-1} A \xrightarrow{\Delta_{(1)}} (\D_{n+1} \circ_{(1)} \D_{n+1}) \circ \sigma^{-1} A \xrightarrow{(Id_{\D_{n+1}}\circ_{(1)}\tau_{n+1})\circ Id_{\sigma^{-1} A}} 
\]
\[
(\D_{n+1} \circ_{(1)} \E_{n+1}) \circ \sigma^{-1} A \cong \D_{n+1} \circ (\E_{n+1} \circ \sigma^{-1} A; \sigma^{-1} A) \xrightarrow{Id_{\D_{n+1}}\circ \gamma_{\sigma^{-1} A}} \D_{n+1} \circ \sigma^{-1} A,
\] 
where $\gamma_{\sigma^{-1} A}$ is the structure map of $\sigma^{-1} A$. We now need to show that $\Upsilon d_2 =d_2 \Upsilon$. This follows from the fact that $\Upsilon $ is a morphism of cooperads, $\sigma^{-1}A$ is a suspension and because the diagram from Proposition \ref{propthecommutingdiagramwiththemapformerlyknownasupsilon} commutes.

 More precisely if we start with $\delta \otimes s^{-1} a_1 \otimes ... \otimes s^{-1} a_r\in \D_{n+1}(r) \otimes_{\S_r} (\sigma^{-1} A)^{\otimes r}\subset B_{\E_{n+1}}\sigma^{-1} A$, then we see that $\Upsilon d_2 (\delta \otimes s^{-1} a_1 \otimes ... \otimes s^{-1} a_r)$ is given by 

\begin{equation}\label{eq1}
\Upsilon d_2 (\delta \otimes s^{-1} a_1 \otimes ... \otimes s^{-1} a_r)= 
\end{equation}
\[
\sum \Upsilon(\delta') \otimes (a_1 \otimes ... \otimes a_{i-1} \otimes \tau_{n+1}(\delta'')(s^{-1} a_i\otimes ... \otimes s^{-1} a_j)\otimes  a_{j+1}\otimes...\otimes a_r), 
\]
where we denote by $\Delta(\delta)=\sum \delta' \otimes \delta''\in \D_{n+1} \circ_{(1)} \D_{n+1}$, the infinitesimal coproduct of $\D_{n+1}$. The other term is given by 
\begin{equation}\label{eq2}
d_2 \Upsilon (\delta \otimes s^{-1} a_1 \otimes ... \otimes s^{-1} a_r)=\sum \Upsilon(\delta') \otimes (a_1 \otimes ... \otimes a_{i-1} \otimes \tau_{n}(\Upsilon(\delta''))( a_i\otimes ... \otimes  a_j)\otimes a_{j+1}\otimes...\otimes a_r),
\end{equation}
where $\sum \Upsilon(\delta') \otimes \Upsilon(\delta'')$ is the infinitesimal coproduct of $\Upsilon(\delta)$. We have use here the fact that $\Upsilon$ is a morphism of cooperads, i.e. $\Delta_{(1)}(\Upsilon(\delta))=\Upsilon \otimes \Upsilon(\Delta_{(1)}(\delta))=\sum \Upsilon(\delta') \otimes \Upsilon (\delta'')$.

Since $\sigma^{-1}$ is a desuspension, and because of Theorem \ref{thrmsuspensionfunctor}, we have an equality between \\ $\tau_{n+1}(\delta'')(s^{-1} a_i\otimes ... \otimes s^{-1} a_j)$ and  $\sigma(\tau_{n+1}(\delta''))( a_i\otimes ... \otimes  a_j)$. Because the diagram of Proposition \ref{propthecommutingdiagramwiththemapformerlyknownasupsilon} commutes we also have the equality $\tau_n \Upsilon=\sigma \tau_{n+1}$. We therefore see that the terms  $\tau_{n+1}(\delta'')(s^{-1} a_i\otimes ... \otimes s^{-1} a_j)$ and $\tau_{n}(\Upsilon(\delta''))( a_i\otimes ... \otimes  a_j)$ are equal. Since the sums both range over the same coproduct and all the terms are equal, the maps $\Upsilon$ and $d_2$ commute. The map $\Upsilon$ is therefore a chain map, which concludes the proof of Proposition \ref{propupsilon}.


\end{proof}

\begin{corollary}\label{corstability}
Let $X$ be an $m$-connected simplicial set and denote by $A$ an $E_{\infty}$-algebra model for $X$. Let $n <m$, then we have a sequence of degree $0$ chain maps 
\[
H_*(B_{\E_n}(A)) \leftarrow s H_*(B_{\E_{n+1}}(\sigma^{-1} A)) \leftarrow s^2 H_*(B_{\E_{n+2}}(\sigma^{-2} A)) \leftarrow s^3 H_*(B_{\E_{n+3}}(\sigma ^{-3} A))   \leftarrow ...
\]

\end{corollary}








\subsection{The relation between $\E_n$- and $\E_{n+1}$-Hopf invariants}



In the previous sections we have shown that the suspension morphism of the $\E_{\infty}$-operad induces a functor from $\E_\infty$-algebras to $\E_\infty$-algebras, which maps a model for a simplicial set $X$ to a model for the suspension of $X$. In particular, for an $\E_\infty$-algebra $A$, this functor induces a map from $[A,S^m]$ to $[\sigma^{-1} A,S^{m+1}]$, the set of homotopy classes of maps between $A$ and a $\E_\infty$-algebra model for the sphere $S^m$ and the suspension of these $\E_\infty$-algebras. We have also seen that the Koszul dual of the suspension morphism induces a map from the set of $\E_{n+1}$-Hopf invariants to the set of $\E_n$-Hopf invariants. In this section we combine these results to obtain a formula which computes the $\E_{n+1}$-Hopf invariants of the suspension of a map $f:S^m \rightarrow X$ in terms of the maps $\Upsilon$ and $\sigma^{-1}$.

The main result is summarized in the following theorem.


\begin{theorem}\label{thrmenen+1formula}
Let $f:S^m\rightarrow X$ be a map from the $m$-dimensional sphere $S^m$ to a simply-connected simplicial set $X$. Let $n<m$ and $\omega \in B_{\E_{n+1}}\sigma^{-1} \N(X)$, then we have the following equality between the $\E_n$-Hopf invariants of $f$ and the $\E_{n+1}$-Hopf invariants of $\Sigma f$,
\[
\left< f, \Upsilon \omega \right>_{\E_n}=\left< \Sigma f, \omega \right>_{\E_{n+1}},
\]
where $\Upsilon$ is the map from Definition \ref{defupsilon} and $\Sigma f$ is the simplicial suspension of the map $f$.
\end{theorem}



\begin{proof}
To prove the theorem we need to show that the evaluation of $\Sigma f^*\omega $ on $S^{m+1}$ is equal to the evaluation $f^* \Upsilon( \omega)$ on $S^m$. 

To do this first observe that we have the following commutative diagram (of chain complexes, not $\E_n$-algebras or $\D_n$-coalgebras):
\[
\xymatrix{
B_{\E_{n+1}}\sigma^{-1} \N(X) \ar[r]^{B_{\E_{n+1}}\Sigma f^*} \ar[d]^{\Upsilon} & B_{\E_{n+1}} \sigma^{-1} \N(S^m) \ar[d]^{\Upsilon} \\
B_{\E_{n}} \N(X) \ar[r]^{B_{\E_n} f^*} & B_{\E_{n}} \N(S^m). 
}
\] 
It is straightforward to check that this diagram commutes. So in particular we have an equality $\Upsilon B_{\E_{n+1}}f^* \omega= B_{\E_{n}}f^* \Upsilon \omega$.

The next step is to notice that the right vertical map $\Upsilon:B_{\E_{n+1}} \sigma^{-1} \N(S^n) \rightarrow B_{\E_{n}} \N(S^m)$ induces an isomorphism between the degree $-(m+1)$ part of $H_{-(m+1)}(B_{\E_{n+1}} \sigma^{-1} \N(S^m))$ and the degree $-m$ part of $H_{-m}(B_{\E_{n}}\N(S^m))$. This is because in weight one, the morphism $\Upsilon$ maps the generator of the cohomology of $S^{m+1}$  to the generator of the cohomology of $S^m$. Since $\Upsilon$ is a chain map and both $H_{-(m+1)}(B_{\E_{n+1}} \sigma^{-1} \N(S^m))$ and $H_{-m}(B_{\E_{n}}\N(S^m))$ are equal to $\K$, this induces an isomorphism between them. 

The next step is to notice that because $\Upsilon$ induces an isomorphism between \\ $H_{-(m+1)}(B_{\E_{n+1}} \sigma \N(S^m))$ and $H_{-m}(B_{\E_{n}}\N(S^m))$, the evaluation of $\omega\in B_{\E_{n+1}} \sigma^{-1} \N(S^m)_{-(m+1)}$ on $S^{m+1}$ is equal to the evaluation of $\Upsilon(\omega)$ on $S^m$. We therefore see that if we start with an a cocycle $\omega\in B_{\E_{n+1}}\sigma^{-1}\N(X)$, that the following two  evaluations are equal
\[
\int_{S^{m+1}} \tau( B_{\E_{n+1}}f^* \omega) =  \int_{S^{m}} \tau( B_{\E_{n}}f^* \Upsilon \omega),
\]
where $\tau ( B_{\E_{n+1}}f^* \omega)$ (resp. $\tau(B_{\E_{n}}f^* \Upsilon \omega)$) is a weight one replacement of the cocycle $(B_{\E_{n+1}}f^* \omega)$ (resp. $B_{\E_{n}}f^* \Upsilon \omega$). Since the left hand side  is the $\E_{n+1}$-Hopf invariant of $\Sigma f$ with respect to the cocycle $\omega$ and the right hand side is the $\E_n$-Hopf invariant of $f$ with respect to the cocycle $\Upsilon \omega$, this equality shows that the formula from Theorem \ref{thrmenen+1formula} indeed holds and concludes its proof.
\end{proof}




\part{Examples, applications and open questions}

In the last part of this paper we will give a few examples and explain how the $\E_n$-Hopf invariants can be used to obtain information about the stable homotopy groups of a space.

But before we continue with some more advanced applications, let us first remark that in theory it should be straightforward to compute the $\E_n$-Hopf invariants of a map between a suitable triangulation of $S^m$ and a finite simplicial complex $X$. At the moment to the author's knowledge there are not many explicit examples of such maps. There are however algorithms that should allow a computer to generate low dimensional examples of such maps (see for example \cite{FFWZ}). This could possibly provide us with more examples of computations of the $\E_n$-Hopf invariants of maps.

As another application, we use the sequence of maps from Corollary \ref{corstability} to apply the $\E_n$-Hopf invariants to stable homotopy theory. More precisely,  there exists  a pairing between the inverse limit of the sequence from Corollary \ref{corstability} and the stable homotopy groups of a space. In Section \ref{secapplycationstostablehomotopytheory}, we describe this pairing in detail.

We end the paper with some open questions.

\section{Examples}\label{secexamples}

In this section we will give a few simple examples of maps with non-zero Hopf invariants. It seems highly likely that it is possible to find more interesting examples, but this would probably require  more advanced techniques.  For now we will only discus some simple  examples and in section \ref{subsecmoreexxaamples} we will explain how to find more interesting examples. But before we give any examples we first need to introduce some notation.

From now on we will exclusively work over $\F_2$, the field with two elements.

\subsection{Some notation}



To describe some of our examples we first need to introduce some notation. All the examples in this paper are of weight two cocycles in the $\E_n$-bar construction. We therefore need to describe the weight one and two components of the $\E_n$-bar construction. 

First we will introduce some notation for the arity two part of the $\E_n$-operad. As is explained in Section \ref{secEnalgebras}, the arity two component of $\E_n$ is spanned by the operations of the form $(Id,\tau,Id,\tau,...)$,  such that the length of the sequence is at most $n+1$ and where $Id$ and $\tau$ are the two elements of $\S_2$. From now on we will denote the binary operation corresponding to the sequence of length $i+1$ by the cup-$i$ product, where $i$ is the length of the sequence minus $2$. For example $(Id,\tau)(\alpha,\beta)$ will be denoted by $\alpha \cup_1 \beta$ and $(\tau,Id,\tau)(\alpha,\beta)$ will be denoted by $\beta \cup_2 \alpha$ (notice the switch between $\alpha$ and $\beta$ because of the symmetric group action). For simplicity we will denote the $\cup_0$-product by just a $\cup$.

As is explained in Sections \ref{secthebarconstruction} and \ref{secEnalgebras}, the underlying graded vector space of the $\E_n$-bar construction of an $\E_n$-algebra $A$ is given by $B_{\E_n}A=\bigoplus_r \D_n(r) \otimes_{\S_r} A^{\otimes r}$. In Section \ref{subsecKoszulduality1}, it is also explained that as a graded vector space $\D_n(2)$ is isomorphic to $\Sus^{-n} \E_n^{\vee}$. A basis for $\D_n(2)$ is therefore given by elements of the form $s^{n}\cup_i^{\vee}$, with $\cup_i\in \E_n(2)$. Note that since $\E_n(2)$ is bounded between degrees $0$ and $n-1$, the arity $2$ part of the cooperad $\D_n(2)$ is bounded between degrees $1$ and $n$ (see Lemmas \ref{lemenisbounded} and \ref{lemDnisbounded}).


We will use the following notation for the weight two elements in the $\E_n$-bar construction. For $0 \leq i \leq n-1$, denote by $\vert_i$ the cooperation in $\D_n(2)_{n-i+1}$ corresponding to the element $s^n \cup_i^{\vee}$. Further denote by  $\nu \vert_i \omega \in \D_n(2) \otimes_{\S_2}  A^{\otimes 2}$, the element defined by $s^{n}\cup_i^{\vee}\otimes \nu \otimes \omega$. The cooperation $\vert_i$ has degree $n-i$ and the differential is given by $d(\vert_i)=\vert_{i+1}+\tau \vert_{i+1}$, where $ \tau$ is the generator of $\S_2$. Since we are working over a field of characteristic $2$ there are no signs. 

Using this notation we can now recall the twisting morphism from Section 6.2.1 of \cite{Fres3}.

\begin{lemma}\label{lemaritytwocompoentoftau}
The arity two component of the twisting morphism $\tau_n:\D_n \rightarrow \E_n$ is defined by sending the element corresponding to the cooperation $\vert_i$ to $\tau^{n-1} \cup_{n-i-1}$, where $\tau$ is the generator of $\S_2$. 
\end{lemma}

\subsection{Some examples}

In all the examples we consider we will implicitly assume that all the spaces we consider are sufficiently triangulated simplicial sets. So in particular we assume that all the maps we consider can be realized as maps of simplicial sets.
 
The first example is the well  known example of the classical Hopf invariant.

\begin{example}\label{exampleprevious}
Let $f:S^{2m-1} \rightarrow S^m$ be a map, then the classical Hopf invariant is defined as the $\E_1$-Hopf invariant with respect to the cocycle $\omega \vert_0 \omega$, where $\omega$ is a cochain representing the generator of the cohomology of $S^m$, such that $\omega \cup \omega=0$. Since $\omega$ is of degree $-m$ and $H^{-m}(S^{2m-1})=0$, we can find a cocycle $d^{-1}f^*\omega$. The weight reduction of $f^*\omega \vert_0 f^*\omega$ is then given by $f^*\omega \vert_0 f^*\omega-d_{B}( d^{-1}f^*\omega \vert_0 f^* \omega)$ which is equal to $d^{-1}f^* \omega \cup f^* \omega$. The corresponding Hopf invariant is then given by
\[
\int_{S^{2m-1}}d^{-1}f^* \omega \cup f^* \omega,
\]
 the evaluation of $d^{-1}f^* \omega \cup f^* \omega$ on the fundamental class of $S^{2m-1}$.  This gives us exactly the formula from Section \ref{secclassicalhopfinvarainaiianiteihskbndsbcoisdebioisgksdaaaaaaaaaaaaaaaaaarrrrrrrgggggggg} and therefore we recover the classical Hopf invariant.
 
 By Adams' famous result (see \cite{Adams1}), there exist only three maps $f_1:S^3 \rightarrow S^2$, $f_2:S^{7} \rightarrow S^4$ and $f_3:S^{15}\rightarrow S^8$ with Hopf invariant $1$ (we exclude the map $f:S^n \rightarrow S^n$). These maps will be important for Example \ref{examplethenextexample}. 
\end{example}

The next example is a generalization of the previous example in which we show that when we work over a field of characteristic $2$ we can detect more classes of maps.

\begin{example}\label{example3}
Let $f:S^{2m-i} \rightarrow S^{m}$ be a map with $i<m$. Let $\omega \in \tilde{N}^{-m}(S^m)$ be a representative of the generator of the cohomology of $S^m$, such that $\omega \cup_{i-1} \omega =0$, further let $i\leq n<m$. We will now give an example of a cocycle that gives a non-zero $\E_n$-Hopf invariant and that has the possibility to detect non-trivial maps from $S^{2m-i}$ to $S^m$. In Example \ref{examplethenextexample} we will show that there actually exist maps with a non-zero Hopf invariant with respect to these cocycles. 

The cocycle we will study in this example is $\omega \vert_{n-i} \omega$ in the $\E_n$-bar construction, because we assumed that $\omega \cup_{i-1} \omega=0$ this is indeed a cocycle. When we pull $\omega \vert_{n-i} \omega$ back along $f$, we get $f^* \omega \vert_{n-i} f^* \omega$, we will now weight reduce this cocycle to get a weight one cocycle which we can evaluate on the fundamental class of $S^{2m-i}$. First note that similar to Example \ref{exampleprevious}, we can find a  $d^{-1}f^*\omega$. The first step in the weight reduction of $\omega \vert_{n-i} \omega$ is then defined as $f^* \omega \vert_{n-i} f^*\omega + d_B d^{-1}( f^*\omega \vert_{n-i} f^*\omega)$, which is equal to $d^{-1} f^*\omega \cup_{i-1} f^*\omega+d^{-1}f^* \omega \vert_{n-i-1} f^* \omega+f^* \omega \vert_{n-i-1} d^{-1}f^* \omega$. This is still a weight $2$ cocycle, but it has a lower $\E_n$-degree. We therefore need to weight reduce again. This time we use the cocycle $d^{-1}f^* \omega \vert_{n-i-1} d^{-1}f^* \omega$. The total weight reduction is then given by $d^{-1}f^* \omega\cup_{i-1}f^* \omega + d^{-1}f^* \omega \cup_{i-2} d^{-1}f^* \omega$. The corresponding Hopf invariant is then given by 
\[
\int_{S^{2m-i}} d^{-1}f^* \omega\cup_{i-1}f^* \omega + d^{-1}f^* \omega \cup_{i-2} d^{-1}f^* \omega.
\]
\end{example}



In the next example we will show how we can use Theorem \ref{thrmenen+1formula} to show that there are examples of maps with non-zero Hopf invariants.

\begin{example}\label{examplethenextexample}
In this example we use the previous examples to show that there are non-trivial maps from $S^{m+i}\rightarrow S^m$ for $i=1,3,7$. It is straightforward to check that with our grading conventions $\Upsilon(\vert_0)=\vert_0$. Then note that the cocycle $\omega \vert_{n-i} \omega$ from Example \ref{example3}, is the image of the cocycle $s^{-1} \omega \vert_{n-i+1} s^{-1} \omega$ under the map $\Upsilon:B_{\E_{n+1}}\N(S^{m+i+1})\rightarrow B_{\E_{n}}\N(S^{m+i})$. For $i=n$, we get a sequence of cocycles 
\[
\omega \vert_0 \omega , \mbox{  } s^{-1} \omega \vert_0 s^{-1} \omega,\mbox{  } ..., \mbox{  } s^{-i} \omega \vert_0 s^{-i} \omega , \mbox{  }... ,
\]
with $s^{-r} \omega \vert_0 s^{-r} \omega \in B_{\E_{n+r}}\sigma^{-r}\N(X)$, such that 
\[
\omega \vert_0 \omega = \Upsilon (s^{-1}\omega \vert_0 s^{-1}\omega)= \Upsilon^2(s^{-2}\omega \vert_0 s^{-2}\omega)= ... = \Upsilon^i(s^{-i}\omega \vert_0 s^{-i}\omega)=...  .
\]
Because $\omega \vert_0 \omega$ is the cocycle corresponding to the classical Hopf invariant and by Adams' result, there exist maps $f_1$, $f_2$ and $f_3$ with Hopf invariant $1$. We can now use Theorem \ref{thrmenen+1formula} which implies that $\Sigma^r f_i :S^{m+i+r} \rightarrow S^{m+r}$ has a non-zero Hopf invariant with respect to the  cocycle  $s^{-r} \omega \vert_0 s^{-r} \omega$. The suspensions of the Hopf invariant one maps are therefore not homotopic to the constant map and we have found an example of a family of maps with non-zero Hopf invariants.
\end{example}

As a corollary we find an alternative proof of the well known result that certain homotopy groups of spheres are non-trivial.

\begin{corollary}
 The homotopy group $\pi_{n+1}(S^n)$ is non-trivial for all $n \geq 2$,  $\pi_{n+3}(S^n)$ is non-trivial for all $n \geq 4$ and $\pi_{n+7}(S^n)$ is non-trivial for all $n\geq 8$ .
\end{corollary}


\begin{remark}
This result is of course already well known and is shown in for example Theorem 4L.2 in \cite{Hat1}. The method used in \cite{Hat1} already has some of the ingredients used in our proof and the theorems in this paper can be seen as a more general and more streamlined version of this idea. 
\end{remark}

\subsection{How to generate more examples}\label{subsecmoreexxaamples}

One of the advantages of the theory described in this paper is, that in theory it would be possible to use a computer to compute the $\E_n$-Hopf invariants of maps between a finite triangulation of the sphere and a finite simplicial complex $X$. Unfortunately there are not many explicit descriptions of such maps, fortunately in \cite{FFWZ} the authors give an algorithm which produces explicit representatives for the homotopy classes of maps from $S^m$ to a simplicial complex $X$. It should therefore be possible to produce a lot more examples of maps and their Hopf invariants by combining this algorithm and the $\E_n$-Hopf invariants.

\section{Applications to stable homotopy theory}\label{secapplycationstostablehomotopytheory}

In this section we explain how we could apply the results from this paper to stable homotopy theory. Recall that the formula from Theorem \ref{thrmenen+1formula} gives a connection between the Hopf invariants of a map and its suspension. We taking the inverse limit of the sequence of maps from Corollary \ref{corstability} we get invariants of stable homotopy classes of maps.


\begin{definition}\label{defBS}
Let $X$ be an $(n+1)$-connected simplicial set, the set of stable $\E_n$-Hopf invariants is defined as the inverse limit of the following sequence
\[
... \xrightarrow{\Upsilon} s^3 B_{\E_{n+3}}\N(\Sigma^3 X)\xrightarrow{\Upsilon} s^2 B_{\E_{n+2}}\N(\Sigma^2 X) \xrightarrow{\Upsilon} s B_{\E_{n+1}}\N(\Sigma X) \xrightarrow{\Upsilon} B_{E_n} \N(X).
\]
We will denote the stable Hopf invariants by $B^S(X):=\varprojlim s^{r} B_{E_{n+r}}\N(\Sigma^r X)$. We will call a cocycle $\omega \in B_{\E_n}\N(X)$ stable if it is the image of an element $\tilde{\omega} \in B^S(X)$ under the canonical map from $B^S(X)$ to  $B_{\E_n}\N(X)$.
\end{definition}

\begin{remark}
Note that in this definition we could replace the $\E_n$-Hopf invariants by $\E_k$-Hopf invariants as long as $k<n$. It is however expected that the higher the $k$ the more the $\E_k$-Hopf invariants detect. The reason we need the connectivity assumption is because this way the Hopf pairing is defined for all degrees $*$.
\end{remark}

Note that because of Theorem \ref{thrmsuspensionfunctor} the sequence from Definition \ref{defBS} is equivalent to the sequence 
\[
... \xrightarrow{\Upsilon} s^3 B_{\E_{n+3}}\sigma^{-3}\N( X)\xrightarrow{\Upsilon} s^2 B_{\E_{n+2}}\sigma^{-2}\N(X) \xrightarrow{\Upsilon} s B_{\E_{n+1}}\sigma^{-1}\N(X) \xrightarrow{\Upsilon} B_{E_n} \N(X).
\]
and we therefore have an isomorphism between $\varprojlim s^r B_{\E_{n+r}}\sigma^{-r}\N(X)$ and $\varprojlim s^{r} B_{E_{n+r}}\N(\Sigma^r X)$. This gives us two different ways to compute $B^S(X)$.

\begin{theorem}\label{thrmBS}
Let $X$ be an $(n+1)$-connected simplicial set, then there exists a pairing
\[
\left<,\right>_S:H_*(B^S(X))\otimes \pi_*^S(X) \rightarrow \K,
\]
which we call the stable Hopf pairing. This pairing is defined as follows. Let $\omega \in B^S(X)_k$ and $f\in \pi_k^S(X)$, then we fix a representative $\tilde{f}:S^{k+r} \rightarrow \Sigma^r X$ of the stable homotopy class $f$ and let $\tilde{\omega} \in B_{\E_{n+r}}\N(X)$ be the image of $\omega$ under the canonical map coming from the limit. Then we define the stable Hopf pairing by
\[
\left<\omega,f \right>_S:=\left<\tilde{\omega},\tilde{f}\right>_{\E_{n+r}}.
\]
\end{theorem}

\begin{proof}
To prove the theorem we need to show that this pairing is independent of $r$ and homotopy invariant. The fact that the pairing is independent of $r$ is a straightforward consequence of Theorem \ref{thrmenen+1formula}. The fact that the pairing is homotopy invariant follows from Theorem \ref{thrmEnhopfinv}, which proves the theorem.
\end{proof}

\begin{corollary}\label{corBS}
Let $f:S^m \rightarrow X$ be a map such that $f$ has a non-zero Hopf invariant with respect to a stable cocycle $\omega\in B_{\E_n}\N(X)$ (with $n<m$), then $f$ induces a non-trivial element of the stable homotopy groups of $X$.
\end{corollary}

\begin{proof}
Since $\omega$ is a stable cocycle there exists a sequence of $\tilde{\omega}_r\in B_{\E_{n+r}}\N(X)$, such that $\Upsilon^r(\tilde{\omega_r})=\omega$ and $\Upsilon(\tilde{\omega}_r)=\tilde{\omega}_{r-1}$. Because of Theorem \ref{thrmenen+1formula}, we have that $\left<\tilde{\omega}_r,\Sigma^r f\right>_{\E_{n+r}}=\left<\Upsilon^r(\tilde{\omega}),f \right>_{\E_n}=\left<\omega,f\right>_{S}$. So because $\left<\omega,f\right>_{\E_n}$ is non-zero, all the suspensions of $f$ will also have a non-zero Hopf invariant. All the suspensions of $f$ are therefore non-trivial and therefore define a non-trivial homotopy class in the stable homotopy groups of $X$.

\end{proof}

For an example of a stable cocycle see Example \ref{examplethenextexample}. One of the nice things about the stable Hopf invariants is that the can detect whether a map is stable before this map is an element of the stable range. For example the Hopf fibration is not in the stable range, but because it has Hopf invariant $1$, it can already be seen that it will determine a non-zero element in the stable homotopy groups of spheres.

\section{Comparison with other approaches}

The Hopf invariant has been a well studied invariant in topology and there are many generalizations of it. The generalization that is closest to the approach describe in this paper, is due to Steenrod in \cite{Steen1}. In this paper he introduces the functional cup-$i$ products which can be seen as the $\E_n$-Hopf invariants coming from the weight two cocycles in the bar construction. Our approach generalizes this in the sense that we do not only use the cup-$i$ products, but also the higher products of the $\E_n$-structure and are therefore able to define much more invariants. It is at the moment still unclear how our approach exactly compares to other generalizations of the Hopf invariant like for example \cite{Steer1967} and \cite{White1950} and this will be the topic of future work.

Another approach for constructing invariants of homotopy classes of maps is by using higher cohomology operations. See for example Chapter 9 of \cite{McC1} for a description of how these invariants are constructed. At the moment it is unclear whether our approach detects more than these higher cohomology operations. However, the invariants described in this paper have a few advantages above the invariants described in Chapter 9 of \cite{McC1}. The first one is that we do not need any auxiliary spaces and the invariants are therefore easier computable. Given a simplicial map  from a finite triangulation of $S^m$ to a finite simplicial complex $X$, the $\E_n$-Hopf invariants should be completely computable by a computer. This would not be possible for the invariants coming from the higher cohomology operations the way they are described in \cite{McC1}. Another advantage of our approach is that the stability results are much easier to prove.

\section{Open questions and future work}

In the last section of this paper we will give some indications for future work and state some open questions. 




\subsection{The twisting morphism $\tau_n:\D_n \rightarrow \E_n$}

For further calculations, it will be important to have an explicit description of the twisting morphism $\tau_n:\D_n \rightarrow \E_n$. The first question is: can we give an explicit description of this morphism? A consequence of this result would be that we could also compute examples of Hopf invariants for fields of characteristic  different than $2$.

One way to circumvent this problem would be to use a different model for the $\D_n$-cooperad. In this paper we have always used the model for $\D_n$ given by $\D_n=\Sus^{-n}\E_n^{\vee}$, one possibility for another model would be $B\E_n$, the operadic bar construction on $\E_n$. This model has the advantage that the twisting morphism $\pi:B\E_n \rightarrow \E_n$ is completely known. The disadvantage is that this cooperad is much larger than the $\D_n$-cooperad we used in this paper. Another disadvantage is that the stability results are also a bit less clear.


\subsection{Stable homotopy groups and the stable Hopf pairing}

The results from Theorem \ref{thrmBS} show that we have a pairing between the stabilization with respect to $\sigma^{-1}$ and the stable homotopy groups of a space. An obvious question is about how strong this pairing is, how much of the stable homotopy type do the stable $\E_n$-Hopf invariants see?  In an unlikely, but best case scenario the stable Hopf invariants would be complete. There is however no concrete evidence that this should be the case and might be mainly wishful thinking.

But even if the stable Hopf invariants are not complete it will be interesting to see how much they see and  it would be interesting to find a concrete description of $H^*(B^S(X))$ for some simple spaces $X$. The first question would be if we can find a basis as a $\K$-vector space for $H^*(B_S(X))$ and further questions are what the algebraic structure on this graded vector space would be.

\subsection{Geometric interpretation}

As is described in Proposition 1.5 of \cite{SW2}, the Hopf invariants coming from the associative bar construction can be interpreted as  cohomology classes of the loop space evaluated on the looping of the fundamental class of the sphere. Or in more detail if $\gamma$ is a cocycle in the associative bar complex on a simplicial set $X$, then  the value of the Hopf invariant is equal to $\gamma(\Omega f^*(\Omega  S^m ))$, the evaluation of $\gamma$ on the image of the looping of the fundamental class of $S^m$. 

A natural question is whether an analog of this interpretation would be true in our setting as well, but now with the $n$-fold loop space instead of the one-fold loop space. An indication that there is such an interpretation, is stated in Proposition 1.3 of \cite{Hu1}, which states that for an $(n+1)$-connected space the cochains on $X$ and the  chains on the $n$-fold loop space are Koszul dual to each other.  There are however several technical differences between that paper and this paper, so some care needs to be taken.

Another question is whether there is an analog of the interpretation of the $\E_n$-Hopf invariants as linking numbers. As is explained in Section 1.2 of \cite{SW2}, the Hopf invariants can be interpreted as generalized linking numbers. A natural question would be if we can find a similar interpretation in this case.

A geometric interpretation could also open up new directions for the computation of the $\E_n$-Hopf invariants. But at the moment it is not yet clear what this would be in practice.



\bibliographystyle{plain}

\bibliography{bibliography}{}

\end{document}